\documentclass[11pt, twoside, leqno]{article}

\usepackage{amssymb}
\usepackage{amsmath}
\usepackage{amsthm}
\usepackage{color}
\usepackage{mathrsfs}

\usepackage{indentfirst}

\usepackage{txfonts}

\allowdisplaybreaks

\def\az{\alpha}

\def\XXint#1#2#3{{\setbox0=\hbox{$#1{#2#3}{\int}$ }
\vcenter{\hbox{$#2#3$ }}\kern-.6\wd0}}

\def\B{{B}_{p,q}^\az}

\def\({\left(}
\def \){ \right)}

 \def\dim{\operatorname{dim}}

\newtheorem{theorem}{Theorem}[section]
\newtheorem{lemma}[theorem]{Lemma}
\newtheorem{corollary}[theorem]{Corollary}
\newtheorem{proposition}[theorem]{Proposition}

\theoremstyle{definition}
\newtheorem{remark}[theorem]{Remark}

\renewcommand{\appendix}{\par
   \setcounter{section}{0}%
   \setcounter{subsection}{0}%
   \setcounter{subsubsection}{0}%
   \gdef\thesection{\@Alph\c@section}%
   \gdef\thesubsection{\@Alph\c@section.\@arabic\c@subsection}%
   \gdef\theHsection{\@Alph\c@section.}%
   \gdef\theHsubsection{\@Alph\c@section.\@arabic\c@subsection}%
   \csname appendixmore\endcsname
 }

\numberwithin{equation}{section}

\begin{document}

\arraycolsep=1pt

\title{\bf\Large An Upbound of Hausdorff's Dimension of the Divergence Set of the fractional Schr\"odinger Operator on $H^s(\mathbb R^n)$
\footnotetext{\hspace{-0.35cm} 2010 {\it
Mathematics Subject Classification}.42B37, 42B15.
\endgraf {\it Key words and phrases.} the Carleson problem, divergence set, the fractional Schr\"{o}dinger operator, Hausdorff dimension, Sobolev space.
}
}
\author{ Dan Li, Junfeng Li\footnote{Corresponding author} ~and Jie Xiao}
\date{}
\maketitle

\vspace{-0.7cm}

\begin{center}
\begin{minipage}{13cm}
{\small {\bf Abstract}\quad
This paper shows
$$
\sup_{f\in H^s(\mathbb{R}^n)}\dim _H\left\{x\in\mathbb{R}^n:\  \lim_{t\rightarrow0}e^{it(-\Delta)^\alpha}f(x)\neq f(x)\right\}\leq n+1-\frac{2(n+1)s}{n}\ \ \text{under}\ \
\begin{cases}
n\geq2;\\
\alpha>\frac12;\\
 \frac{n}{2(n+1)}<s\leq\frac{n}{2} .
\end{cases}
$$	

}
\end{minipage}
\end{center}


\section{Introduction}\label{s1}
\subsection{Statement of Theorem \ref{theorem 1.1}}\label{s11}
 From now on, suppose that $\mathcal{S}(\mathbb R^n)$ is the Schwartz space of all functions $f:\mathbb R^n\to \mathbb C$ such that $$
f\in C^\infty(\mathbb R^n)\ \  \&\ \ \lim_{|x|\to\infty}x^\beta\partial^\gamma f(x)=0\ \ \forall\ \ \text{multi-indices}\ \ \beta,\gamma.
$$
Also, let $H^s(\mathbb{R}^n)$ be the $\mathbb R\ni s$-Sobolev space of all tempered distributions $f\in \mathcal{S}'(\mathbb R^n)$ whose Fourier transforms $\hat{f}$ obey
$$\|f\|_{H^s(\mathbb{R}^n)}=\left(\int_{\mathbb{R}^n} \left(1+|\xi|^2\right)^s \left|\hat{f}(\xi)\right|^2d\xi\right)^{\frac{1}{2}}<\infty.
$$
If $(-\Delta)^{\alpha}f$ stands for the $(0,\infty)\ni\alpha$-pseudo-differential operator defined by the Fourier transformation acting on $f\in \mathcal{S}'(\mathbb R^n)$:
$$
((-\Delta)^\alpha f)^{\wedge}(x)=|x|^{2\alpha}\hat{f}(x)\ \ \forall\ \ x\in\mathbb R^n,
$$
then
\begin{align}\label{Boussinesq}
u(x,t)
=e^{it(-\Delta)^\alpha}f(x)
=(2\pi)^{-n}\int_{\mathbb{R}^n} e^{ix\cdot\xi}e^{it|\xi|^{2\alpha}}\hat{f}(\xi)d\xi
\end{align}
exists as a distributional solution to
the $\alpha$-Schr\"odinger equation:
\begin{equation}\label{Schrodinger equation}
\left\{
\begin{aligned}
&\big(i\partial_t+(-\Delta)^{\alpha}\big)u(x,t)=0\ \ \forall\ \  (x,t)\in\mathbb{R}^n\times\mathbb{R};\\
&u(\cdot,0)=f(\cdot)\in H^s(\mathbb R^n). \\
\end{aligned}
\right.
\end{equation}

While understanding the  Carleson problem of deciding such a critical regularity number $s_c$ that
\begin{align}\label{ConvergenceS}
\lim_{t\rightarrow0}e^{it(-\Delta)^\alpha}f(x)=f(x)\ \ \text{a.e.}\ \  x\in\mathbb{R}^n
\ \
\text{holds for all}\ \ f\in H^{s}(\mathbb{R}^n) \ \  \& \ \ s>s_c,
\end{align}
we are suggested to determine the Hausdorff dimension of the divergence set of the $\alpha$-Schr\"odinger operator $e^{it(-\Delta)^{\alpha}}f(x)$:
\begin{align}\label{Dimension}
\mathsf{d}(s,n,\alpha)=\sup_{f\in H^s(\mathbb{R}^n)}\dim _H\left\{x\in\mathbb{R}^n: \lim_{t\rightarrow0}e^{it(-\Delta)^\alpha}f(x)\neq f(x)\right\},
\end{align}
thereby discovering the case $\alpha>\frac12$:

\begin{theorem}
	\label{theorem 1.1}
\begin{equation}\label{s}
\mathsf{d}(s,n,\alpha)\leq n+1-\frac{2(n+1)s}{n}\ \ \text{under}\ \ n\geq2\ \  \& \ \  \alpha>\frac12 \ \  \& \ \ \frac{n}{2(n+1)}<s\leq\frac{n}{2}  .
\end{equation}
\end{theorem}

\subsection{Relevance of Theorem \ref{theorem 1.1}}\label{s12} Here, it is appropriate to say more words on evaluating $\mathsf{d}(s,n,\alpha)$.

\begin{itemize}
\item[$\rhd$] In general, we have the following development.

\begin{itemize}
\item Theorem \ref{theorem 1.1} actually recovers Cho-Ko's \cite{CK} a.e.-convergence result:
$$
f\in H^{s}(\mathbb{R}^n) \ \ \& \ \ s>\frac{n}{2(n+1)}\Rightarrow
\lim_{t\rightarrow0}e^{it(-\Delta)^\alpha}f(x)=f(x)\quad\textrm{a}.\textrm{e}.~ x\in\mathbb{R}^n.
$$

\item A trivial part of Theorem \ref{theorem 1.1} reveals:
$$
\|f\|_{L^\frac{2n}{n-2\alpha}(\mathbb R^n)}\lesssim \|f\|_{H^s(\mathbb R^n)}
\Rightarrow
\mathsf{d}(s,n<2s,\alpha)=0.
$$
Moreover, Theorem \ref{theorem 1.1} improves \eqref{BBCR} under
$$
\frac{n}{2(n+1)}<s\le\frac{n+1}{4},
$$
as stated below:
\begin{itemize}
	\item In \cite{SS}
	Sj\"{o}gren-Sj\"{o}lin showed
	\begin{equation}
	\label{eSS}
	\mathsf{d}(s,n,\alpha)<n+1-2s\ \ \text{as}\ \ \frac{1}{2}<s\leq\frac{n}{2} \ \  \& \ \  \alpha>\frac12.
	\end{equation}
	\item
	In \cite{BBCR} and \cite{Z} it was proved by  Barcel\'{o}-Bennett-Carbery-Rogers and \v{Z}ubrini\'{c} that
	\begin{align}\label{ZBBCR}
	\mathsf{d}(s,n,\alpha)=n-2s\ \ \text{as}\ \ \frac{n}{4}\leq s\leq\frac{n}{2}.
	\end{align}
	
	 \item In \cite{BBCR} Barcel\'{o}-Bennett-Carbery-Rogers gave
	\begin{equation}\label{BBCR}
	\mathsf{d}(s,n,\alpha)\leq
	\begin{cases}
	n+1-2s\ \ \text{as}\ \  \frac{1}{2}<s\leq\frac{n}{4}; \\
	\frac{3n}{2}+1-4s\ \ \text{as}\ \ \frac{n}{4}<s\leq\frac{n+1}{4};\\
	n-2s\ \ \text{as}\ \ \frac{n+1}{4}<s\leq\frac{n}{2}.
	\end{cases}
	\end{equation}
\end{itemize}
\end{itemize}

\item[$\rhd$] In particular, we have the following case-by-case treatment.
\begin{itemize}

\item {\it Case} $\alpha=1$. Under this setting, Theorem \ref{theorem 1.1} coincides with Du-Zhang's \cite[Theorem 2.4]{DZ} since \eqref{Boussinesq} turns out to be the classical Sch\"{o}dinger operator $e^{-it\Delta}f(x)$. \eqref{ConvergenceS} was first proposed in \cite{C} by Carleson for this special case, and then intensively studied in e.g. \cite{B3,B1,B2,L,MYZ,MVV,S1,S3,TV,V2,V1}. Upon combining the results in \cite{C,DK,B2,DGL,DZ}, we conclude $s_{c}=\frac{n}{2(n+1)}$.
Furthermore, in \cite{SS} Sj\"{o}gren-Sj\"{o}lin considered $\mathsf{d}(s,n,1)$. Note that the Sobolev embedding ensures $\mathsf{d}(s,n<2s,1)=0$. So it is enough to calculate $\mathsf{d}(s,n\ge 2s,1)$.
\begin{itemize}
\item
Bourgain's counterexample in \cite{B2} and
Luc\`{a}-Rogers' result in \cite{LR2} showed
$$\mathsf{d}(s,n,1)= n\ \ \text{as}\ \ s\leq\frac{n}{2(n+1)}.
$$
\item The results in \v{Z}ubrini\'{c} \cite{Z} and Barcel\'{o}-Bennett-Carbery-Rogers \cite{BBCR} found $$
\mathsf{d}(s,n,1)=n-2s\ \ \text{as}\ \ \frac{n}{4}\le s\le\frac{n}{2}.
$$
Accordingly,
$$ ~\frac{n}{2(n+1)}=\frac{n}{4}=\frac14\Rightarrow\mathsf{d}(s,1,1)=1-2s.
$$
\item On the one hand, in \cite{DZ} Du-Zhang proved
$$\mathsf{d}(s,n,1)\leq n+1-\frac{2(n+1)s}{n} \ \ \text{as}\ \ \frac{n}{2(n+1)}<s<\frac{n}{4}\ \ \&  \ \ n\geq2.$$
On the other hand, in \cite{LR2, LR1} Luc\`{a}-Rogers obtained
$$
\mathsf{d}(s,n,1)\geq
\left\{
\begin{aligned}
 &n+\frac{n}{n-1}-\frac{2(n+1)s}{n-1} \ \ \text{as}\ \ \frac{n}{2(n+1)}\leq s<\frac{n+1}{8};\\
&n+1-\frac{2(n+2)s}{n}\ \ \text{as}\ \  \frac{n+1}{8}\leq s<\frac{n}{4}. \\
\end{aligned}
\right.
$$
Thus there is still a gap to determine the exact value of $\mathsf{d}(s,n,1)$; see also \cite{DGLZ,DZ,LR3,LR1,LR2} for more information.
\end{itemize}

\item {\it Case} $\alpha\in (2^{-1},\infty)$. Sj\"{o}lin \cite{S1} proved  $s_c=2^{-2}$ for $n=1$. By the iterative argument developed in \cite{B1}, Miao-Yang-Zheng \cite{MYZ} proved that \eqref{ConvergenceS} holds for $$
s>\frac{3}{8}\ \ \&\ \ n=2.
$$ Very recently, Cho-Ko \cite{CK} proved that \eqref{ConvergenceS} holds for $$s>\frac{n}{2(n+1)}\ \ \&\ \ n\ge 2.
$$ It seems that the case $\alpha>2^{-1}$ shares the same critical index with the case $\alpha=1$. So far there has been no counterexample to verify this problem.

\item {\it Case} $\alpha\in (0,2^{-1}]$. It is uncertain that Theorem \ref{theorem 1.1} can be extended to the fractional Schr\"odinger operator $e^{it(-\Delta)^{\alpha}}f(x) \ \ \& \ \  0<\alpha\leq 2^{-1}$. So, an investigation of this extension coupled with the foregoing counterexample will be the subject of future articles.
\end{itemize}
\end{itemize}

In the sequel of this paper, we always assume $\alpha>\frac12$.

 In \S\ref{s2}, we verify Theorem \ref{theorem 1.1} via Proposition \ref{theorem 1.3} \& Theorem \ref{theorem 1.5} - a global $L^{1}$ \& a local $L^2$ estimates for the maximal operator living on a compactly-supported Borel measure and $e^{it(-\Delta)^{\alpha}}f(x)$. However, the proof of Theorem \ref{theorem 1.5} is given in \S \ref{s3} via Theorem \ref{theorem 2.1} - an $L^{\frac{2(n+1)}{n-1}}$-estimate for $e^{it(-\Delta)^{\alpha}}f(x)$ and its Corollary \ref{corollary 2.2} - an $L^2$-estimate for $e^{it(-\Delta)^{\alpha}}f(x)$. Thanks to a highly nontrivial analysis, \S \ref{s4} is devoted to presenting a proof of Theorem \ref{theorem 2.1} which essentially relies on Theorems \ref{theorem 3.1}\&\ref{theorem 3.4} - the broad $1\le k\le n+1$ linear refined Strichartz estimates in dimension $n+1$ and Lemma \ref{lemma 3.6} - the narrow $L^\frac{2(n+1)}{n-1}$-estimate for $ e^{it(-\Delta)^{\alpha}}f(x)$.

\smallskip

\noindent {\it Notation.} In what follows, $A\lesssim B$ stands for $A\le C B$ for a constant $C>0$ and $A\thicksim B$ means  $A\lesssim B\lesssim A$. Further more, for given large number $R$ and small enough $0<\epsilon<1$,  $A\lessapprox B$ stands for $A\le C  R^\epsilon B$ for a constant $C>0$ and $A\thickapprox B$ means  $A\lessapprox B\lessapprox A$.

\section{Theorem \ref{theorem 1.5}$\Rightarrow$ Theorem \ref{theorem 1.1}}\label{s2}
\subsection{Proposition \ref{theorem 1.3} \& its proof} In order to determine the Hausdorff dimension of the divergence set of $e^{it(-\Delta)^\alpha}f(x)$, we need a law for $H^s(\mathbb R^n)$ to be embedded into $L^1(\mu)$ with a lower dimensional Borel measure $\mu$ on $\mathbb R^n$.

\begin{proposition}\label{theorem 1.3}
For a nonnegative Borel measure $\mu$ on $\mathbb R^n$ and $0\le\kappa\le n$, let
$$
C_\kappa(\mu)=\sup_{(x,r)\in\mathbb{R}^n\times(0,\infty)}r^{-\kappa}{\mu\big(B^n(x,r)\big)}\ \ \text{with}\ \ B^n(x,r)=\{y\in\mathbb R^n: |y-x|<r\}
$$
and $M^\kappa(\mathbb{B}^n)$ be the class of all probability measures $\mu$ with $C_\kappa(\mu)<\infty$ and being supported in the unit ball $\mathbb{B}^n=B^n(0,1)$.
Suppose
$$
\begin{cases}
0<s\le\frac{n}{2};\\
\kappa>\kappa_0\ge n-2s;\\
(N,f,\mu)\in [1,\infty)\times H^s(\mathbb R^n)\times M^\kappa(\mathbb B^n);\\
\psi(r)=\exp(-r^2);\\
e_N^{it(-\Delta)^\alpha}f(x)=(2\pi)^{-n}\int_{\mathbb{R}^n} \psi\left(\frac{|\xi|}{N}\right)e^{i(x\cdot\xi+t|\xi|^{2\alpha})}\hat{f}(\xi)d\xi.
\end{cases}
$$
\begin{itemize}
\item[\rm (i)] If $t\in\mathbb R$, then

\begin{align}\label{1.3b1}
\left\|\sup_{1\le N<\infty }\left|e_N^{it(-\Delta)^\alpha}f\right|\right\|_{L^1(\mu)}\lesssim\sqrt{C_\kappa(\mu)}\|f\|_{H^s(\mathbb{R}^n)}.
\end{align}
\item[\rm (ii)] If
\begin{equation}\label{eq:1} \left\|\sup_{0<t<1}\left|e^{it(-\Delta)^\alpha}f\right|\right\|_{L^1(\mu)}\lesssim \sqrt{C_\kappa(\mu)}\|f\|_{H^s(\mathbb{R}^n)},
\end{equation}
then $\mathsf{d}(s,n,\alpha)\leq \kappa_0$.
\end{itemize}
\end{proposition}

\begin{proof} (i) This \eqref{1.3b1} is the elementary stopping-time-maximal inequality \cite[(4)]{BBCR}.
	
	(ii) The argument is split into two steps.

\begin{itemize}
\item[Step 1.]
We show the following inequality:
\begin{align}\label{1.3c1}
\left\|\sup_{0<t<1}\sup_{N\geq1}\left|e_N^{it(-\Delta)^\alpha}f\right|\right\|_{L^1(\mu)}\lesssim\sqrt{C_\kappa(\mu)}\|f\|_{H^s(\mathbb{R}^n)}.
\end{align}
In a similar way to verify \cite[Proposition 3.2]{BBCR}, we achieve
$$\sup_{N\geq1}\left|e_N^{it(-\Delta)^\alpha}f(x)\right|
\leq\left|e_1^{it(-\Delta)^\alpha}f(x)\right|+\int_1^\infty\left|\frac{d}{dN}e_N^{it(-\Delta)^\alpha}f(x)\right|dN.$$

It is not hard to obtain \eqref{1.3c1} if we have the following two inequalities:
\begin{align}\label{1.3c2}
\left\|\sup_{0<t<1}\left|e_1^{it(-\Delta)^\alpha}f\right|\right\|_{L^1(\mu)}\lesssim\sqrt{C_\kappa(\mu)}\|f\|_{H^s(\mathbb{R}^n)}
\end{align}
and
\begin{align}\label{1.3c3}
\int_1^\infty\left\|\sup_{0<t<1}\left|e^{it(-\Delta)^\alpha}\left(\frac{(\cdot)}{N^2}\psi ' \left(\frac{(\cdot)}{N}\right)\hat{f}(\cdot)\right)^{\vee}\right|\right\|_{L^1(\mu)}dN
\lesssim\sqrt{C_\kappa(\mu)}\|f\|_{H^s(\mathbb{R}^n)}.
\end{align}

\eqref{1.3c2} follows from the fact that \eqref{eq:1} implies
\begin{align*}
\left\|\sup_{0<t<1}\left|e_1^{it(-\Delta)^\alpha}f\right|\right\|_{L^1(\mu)}
&=\left\|\sup_{0<t<1}\left|\int_{\mathbb{R}^n}e^{i(x\cdot\xi+t|\xi|^{2\alpha})}\psi(\xi)\hat{f}(\xi)d\xi\right|\right\|_{L^1(\mu)}\\
&=\left\|\sup_{0<t<1}\left|e^{it(-\Delta)^\alpha}\left(\psi(\cdot)\hat{f}(\cdot)\right)^{\vee}\right|\right\|_{L^1(\mu)}\nonumber\\
&\lesssim\sqrt{C_\kappa(\mu)}\left\|\left(\psi(\cdot)\hat{f}(\cdot)\right)^{\vee}\right\|_{H^s(\mathbb{R}^n)}\nonumber\\
&\lesssim\sqrt{C_\kappa(\mu)}\|f\|_{H^s(\mathbb{R}^n)}.\nonumber
\end{align*}

To prove \eqref{1.3c3}, we utilize $$\psi'\left(\frac{|\xi|}{N}\right)\lesssim\sum_{k\geq0}2^{-2nk}\chi_{B^n(0,2^kN)}(\xi)
$$
to calculate
\begin{align}\label{1.3c5}
\left\|\left(\frac{\psi'\left(\frac{(\cdot)}{N}\right)(\cdot)\hat{f}(\cdot)}{N^2}\right)^{\vee}\right\|_{H^s(\mathbb{R}^n)}
&\lesssim\left\|\frac{(1+|\cdot|^2)^\frac{s}{2}\sum_{k\geq0}2^{-2nk}\chi_{B^n(0,2^kN)}(\cdot)(\cdot)\hat{f}(\cdot)}{N^2}\right\|_{L^2(\mathbb{R}^n)}\\
&\leq\sum_{k\geq0}\frac{2^{-2nk}}{N^{1+\epsilon}}
\left\|\frac{(1+|\cdot|^2)^\frac{s}{2} \chi_{B^n(0,2^kN)}(\cdot)(\cdot)\hat{f}(\cdot)}{N^{1-\epsilon}}\right\|_{L^2(\mathbb{R}^n)}\nonumber\\
&\lesssim\frac{1}{N^{1+\epsilon}}\left\|f\right\|_{H^{s+\epsilon}(\mathbb{R}^n)}.\nonumber
\end{align}
By \eqref{eq:1} and \eqref{1.3c5}, we obtain
\begin{align*}
\int_1^\infty\left\|\sup_{0<t<1}\left|e^{it(-\Delta)^\alpha}\left(\frac{\psi' \left(\frac{(\cdot)}{N}\right)(\cdot)\hat{f}(\cdot)}{N^2}\right)^{\vee}\right|\right\|_{L^1(\mu)}dN
&\lesssim\int_1^\infty \sqrt{C_\kappa(\mu)} \left\|\left(\frac{\psi'\left(\frac{(\cdot)}{N}\right)(\cdot)\hat{f}(\cdot)}{N^2}\right)^{\vee}\right\|_{H^s(\mathbb{R}^n)}dN\\
&\lesssim\int_1^\infty \sqrt{C_\kappa(\mu)} \frac{1}{N^{1+\epsilon}}\left\|f\right\|_{H^{s+\epsilon}(\mathbb{R}^n)}dN\nonumber\\
&\lesssim\sqrt{C_\kappa(\mu)}\|f\|_{H^{s+\epsilon}(\mathbb{R}^n)},\nonumber
\end{align*}
thereby reaching \eqref{1.3c3}.
\end{itemize}

\begin{itemize}
\item[Step 2.]
We are about to show:
$$
\mathsf{d}(s,n,\alpha)\leq\kappa_0\ \ \forall\ \ \kappa_0\in[n-2s,\kappa).
$$
By the definition, we have
\begin{align}\label{1.3d2}
\mu\left\{x\in\mathbb{B}^n: \lim_{t\rightarrow0}e^{it(-\Delta)^\alpha}f(x)\neq f(x)\right\}
&=\mu\left\{x\in\mathbb{B}^n: \lim_{t\rightarrow0}\lim_{N\rightarrow\infty}e_N^{it(-\Delta)^\alpha}f(x)\neq \lim_{N\rightarrow\infty}e_N^{i0(-\Delta)^\alpha}f(x)\right\}.
\end{align}
For any $$f\in H^s(\mathbb R^n)\ \ \&\ \ 0<\epsilon\ll 1,$$ there exists
$$g\in \mathcal S(\mathbb{R}^n)\ \ \text{such that}\ \
\|f-g\|_{H^s(\mathbb{R}^n)}<\epsilon.$$
Accordingly, if
$$\mu\in M^\kappa(\mathbb{B}^n)~\&~\kappa>\kappa_0\geq n-2s,$$
then a combination of \eqref{1.3c1} and \eqref{1.3b1} gives
\begin{align}\label{1.3d3}
&\mu\left\{x\in\mathbb{B}^n: \overline{\lim}_{t\rightarrow0}\overline{\lim}_{N\rightarrow\infty}
\left|e_N^{it(-\Delta)^\alpha}f(x)-e_N^{i0(-\Delta)^\alpha}f(x)\right|>\lambda\right\}\\
&\leq\mu\left\{x\in\mathbb{B}^n: \sup_{0<t<1}\sup_{N\geq1}\left|e_N^{it(-\Delta)^\alpha}(f-g)(x)\right|>\frac{\lambda}{3}\right\}\nonumber\\
&\quad+\mu\left\{x\in\mathbb{B}^n: \lim_{t\rightarrow0}\lim_{N\rightarrow\infty}\left|e_N^{it(-\Delta)^\alpha}g(x)-e_N^{i0(-\Delta)^\alpha}g(x)\right|>\frac{\lambda}{3}\right\}\nonumber\\
&\quad+\mu\left\{x\in\mathbb{B}^n: \sup_{N\geq1}\left|e_N^{i0(-\Delta)^\alpha}(g-f)(x)\right|>\frac{\lambda}{3}\right\}\nonumber\\
&\leq\lambda^{-1}\left\|\sup_{0<t<1}\sup_{N\geq1}\left|e_N^{it(-\Delta)^\alpha}(f-g)\right|\right\|_{L^1(\mu)}
+\lambda^{-1}\left\|\sup_{N\geq1}\left|e_N^{i0(-\Delta)^\alpha}(g-f)\right|\right\|_{L^1(\mu)}\nonumber\\
&\lesssim\lambda^{-1}\sqrt{C_\kappa(\mu)}\|f-g\|_{H^s(\mathbb{R}^n)}\nonumber\\
&\lesssim\lambda^{-1}\sqrt{C_\kappa(\mu)}\epsilon.\nonumber
\end{align}
Upon letting $\epsilon\rightarrow0$ firstly and $\lambda\rightarrow0$ secondly, we have
$$\mu\left\{x\in\mathbb{B}^n: \lim_{t\rightarrow0}  e^{it(-\Delta)^\alpha}f(x)\neq f(x) \right\}=0.$$

If $\mathbb{H}^\kappa$ denotes the $\kappa$-dimensional Hausdorff measure which is of translation invariance and countable additivity, then
Frostman's lemma is used to derive
$$\mathbb{H}^\kappa\left\{x\in\mathbb{B}^n: \lim_{t\rightarrow0}e^{it(-\Delta)^\alpha}f(x)\neq f(x)\right\}=0,$$
and hence
$$\mathsf{d}(s,n,\alpha)=\sup_{f\in H^s(\mathbb{R}^n)}\dim_H\left\{x\in\mathbb{R}^n: \lim_{t\rightarrow0}e^{it(-\Delta)^\alpha}f(x)\neq f(x)\right\}\leq\kappa_0.
$$

\end{itemize}

\end{proof}

\subsection{Proof of Theorem \ref{theorem 1.1}}\label{s22} We begin with a statement of the following key result whose proof will be presented in \S \ref{s3} due to its high nontriviality.

\begin{theorem}\label{theorem 1.5}
	If
	\begin{equation*}
	\left\{
	\begin{aligned}
	&n\geq2;\\
	&0<\kappa\le n;\\
	&C_\kappa(\mu)<\infty;\\
	&R\geq1;\\
	&d\mu_R(x)=R^\kappa d\mu\left(\frac{x}{R}\right);\\
	&f\in H^s(\mathbb R^n);\\
	&\text{supp}\hat{f}\subset A(1)=\{\xi\in\mathbb{R}^n:|\xi|\thicksim1\},\\
	\end{aligned}
	\right.
	\end{equation*}
	then
	\begin{align}\label{1.5a}
	\left\|\sup_{0<t<R}\left|e^{it(-\Delta)^\alpha}f\right|\right\|_{L^2(B^n(0,R);\mu_R)}
	\lessapprox R^{\frac{\kappa}{2(n+1)}}\|f\|_{L^2(\mathbb{R}^n)}.
	\end{align}
\end{theorem}

Consequently, we have the following assertion.

\begin{corollary}\label{theorem 1.4}
If
\begin{equation*}
\left\{
\begin{aligned}
 &n\geq2;\\
 &0<\kappa\le n;\\
 &s>\frac{\kappa}{2(n+1)}+\frac{n-\kappa}{2};\\
 &C_\kappa(\mu)<\infty;\\
&f\in H^s(\mathbb{R}^n),
\end{aligned}
\right.
\end{equation*}
then
\begin{align}\label{1.4a}
\left\|\sup_{0<t<1}\left|e^{it(-\Delta)^\alpha}f\right|\right\|_{L^2(\mathbb B^n;\mu)}\leq \sqrt{C_\kappa(\mu)}\|f\|_{H^s(\mathbb{R}^n)}.
\end{align}
\end{corollary}

\begin{proof}
	Upon using Theorem \ref{theorem 1.5} and its notations as well as \cite{CK} (cf. \cite{CLV,L,LR,MYZ}), we get
	\begin{align}\label{1.4b1}
	\left\|\sup_{0<t<R^{2\alpha}}\left|e^{it(-\Delta)^\alpha}f\right|\right\|_{L^2(B^n(0,R);\mu_R(x))}
	\lessapprox R^{\frac{\kappa}{2(n+1)}}\|f\|_{L^2(\mathbb{R}^n)}.
	\end{align}
	
	Next, we use parabolic rescaling. More precisely, if
	
	\begin{equation*}
	\left\{
	\begin{aligned}
	&\xi=R^{-1}\eta;\\
	&x=RX; \\
	&t=R^{2\alpha} T;\\
	&f_R(x)=f(Rx);\\
	&\text{supp} \widehat{f_R}\subset A(R)=\{\xi\in\mathbb{R}^n:|\xi|\thicksim R\},
	\end{aligned}
	\right.
	\end{equation*}
	then
	\begin{align*}
	e^{it(-\Delta)^\alpha}f(x)
	&=\int_{\mathbb{R}^n}e^{i(x\cdot\xi+t|\xi|^{2\alpha})}\hat{f}(\xi)d\xi\\
	&=\int_{\mathbb{R}^n}e^{i(R^{-1}x\cdot\eta+tR^{-2\alpha}|\eta|^{2\alpha} )}\widehat{f(R\cdot)}(\eta)d\eta\\
	&=\int_{\mathbb{R}^n}e^{i(X\cdot\eta+T|\eta|^{2\alpha})}\widehat{f_R}(\eta)d\eta\\
	&=e^{iT(-\Delta)^{\alpha}}f_R(X),
	\end{align*}
	and hence
	$$
	\begin{cases}
	\left\|\sup_{0<t<R^{2\alpha}}\left|e^{it(-\Delta)^\alpha}f\right|\right\|_{L^2(B^n(0,R);\mu_R(x))}=R^{\frac{\kappa}{2}}\left\|\sup_{0<T<1}\left|e^{iT(-\Delta)^{\alpha}}f_R\right|\right\|_{L^2(\mathbb B^n;\mu(X))};\\
	\|f_R\|_{L^2(\mathbb{R}^n)}=\left(\int_{\mathbb{R}^n}|f_R(x)|^2dx\right)^{\frac{1}{2}}=R^{-\frac{n}{2}}\|f\|_{L^2(\mathbb{R}^n)};\\
	R^{\frac{\kappa}{2}}\left\|\sup_{0<T<1}\left|e^{iT(-\Delta)^{\alpha}}f_R\right|\right\|_{L^2(\mathbb B^n;d\mu(X))}
	\lessapprox  R^{\frac{\kappa}{2(n+1)}}R^{\frac{n}{2}}\|f_R\|_{L^2(\mathbb{R}^n)}.
	\end{cases}
	$$
	Consequently, if
	$T=t\ \&\ X=x$, then
	\begin{align}\label{1.4b2}
	\left\|\sup_{0<t<1}\left|e^{it(-\Delta)^\alpha}f_R\right|\right\|_{L^2(\mathbb B^n;d\mu(x))}
	\lessapprox  R^{\frac{\kappa}{2(n+1)}+\frac{n-\kappa}{2}}\|f_R\|_{L^2(\mathbb{R}^n)},
	\end{align}
	and hence Littlewood-Paley's decomposition yields
	
	\begin{equation*}
	\left\{
	\begin{aligned}
	&f=f_0+\sum_{k\geq1}f_k;\\
	&\text{supp} \widehat{f_0}\subset A(1); \\
	&\text{supp} \widehat{f_k}\subset A(2^k)=\{\xi\in\mathbb{R}^n:|\xi|\thicksim 2^k\}.\\
	\end{aligned}
	\right.
	\end{equation*}
	Finally, by Minkowski's inequality and \eqref{1.4b2} as well as $$s>\frac{\kappa}{2(n+1)}+\frac{n-\kappa}{2},$$
	we arrive at
	\begin{align*}
	\left\|\sup_{0<t<1}\left|e^{it(-\Delta)^\alpha}f\right|\right\|_{L^2(\mathbb B^n;\mu)}
	&\leq\left\|\sup_{0<t<1}\left|e^{it(-\Delta)^\alpha}f_0\right|\right\|_{L^2(\mathbb B^n;\mu)}
	+\sum_{k\geq1}\left\|\sup_{0<t<1}\left|e^{it(-\Delta)^\alpha}f_k\right|\right\|_{L^2(\mathbb B^n;\mu)}\\
	&\lesssim\|f_0\|_{L^2(\mathbb{R}^n)}
	+\sum_{k\geq1}2^{k\left({\frac{\kappa}{2(n+1)}}+\frac{n-\kappa}{2}\right)}\|f_k\|_{L^2(\mathbb{R}^n)}\\
	&\lesssim\|f\|_{H^s(\mathbb{R}^n)}
	+\sum_{k\geq1}2^{k\left({\frac{\kappa}{2(n+1)}}+\frac{n-\kappa}{2}-s\right)}\|f\|_{H^s(\mathbb{R}^n)}\\
	&\lesssim\|f\|_{H^s(\mathbb{R}^n)}.
	\end{align*}
	
\end{proof}

\begin{proof}[Proof of (Corollary \ref{theorem 1.4}$\Rightarrow$Theorem\ref{theorem 1.1})]  An application of the H\"older inequality and  \eqref{1.4a} in Corollary \ref{theorem 1.4} derives
	
	\begin{align*}\label{1.4a}
	\left\|\sup_{0<t<1}\left|e^{it(-\Delta)^\alpha}f\right|\right\|_{L^1(\mathbb B^n;\mu)}\leq \sqrt{C_\kappa(\mu)}\|f\|_{H^s(\mathbb{R}^n)},
	\end{align*}
	whence \eqref{eq:1} follows up. So,
	Proposition \ref{theorem 1.3} yields
	$$
	\mathsf{d}(s,n,\alpha)\le\kappa_0\in [n-2s,\kappa).
	$$
	
Also, since
$$s>\frac{\kappa}{2(n+1)}+\frac{n-\kappa}{2},$$
we have
$$n\ge\kappa>n+1-\frac{2(n+1)s}{n}.$$
Upon choosing
$$\kappa_0=n+1-\frac{2(n+1)s}{n},
$$
we make a two-fold analysis below:
\begin{itemize}
	\item[$\rhd$]
	On the one hand, we ask for
	$$
	n+1-\frac{2(n+1)s}{n}\geq n-2s\Leftrightarrow s\leq\frac{n}{2}.
	$$
	
	\item[$\rhd$]
	On the other hand, it is nature to request
	$$n+1-\frac{2(n+1)s}{n}<n\Leftrightarrow s>\frac{n}{2(n+1)}.$$
\end{itemize}
Accordingly,
$$
\frac{n}{2(n+1)}<s\leq\frac{n}{2}
$$
is required in the hypothesis of Theorem \ref{theorem 1.1}.
\end{proof}

\section{Theorem \ref{theorem 2.1} $\Rightarrow$ Theorem \ref{theorem 1.5} }\label{s3}

\subsection{Theorem \ref{theorem 2.1} $\Rightarrow$ Corollary \ref{corollary 2.2}}
 We say that a collection of quantities are dyadically constant if all the quantities are in the same interval of the form $(2^j,2^{j+1}]$, where $j$ is an integer. The key ingredient of the proof of Theorem \ref{theorem 1.5} is the following Theorem \ref{theorem 2.1} which will be proved in \S \ref{s4}.
\begin{theorem}\label{theorem 2.1}
Let
\begin{equation*}
\left\{
\begin{aligned}
&(n, R)\in\mathbb N \times[1,\infty);\\
&\text{supp}\hat{f}\subset \mathbb B^n;\\
&p=\frac{2(n+1)}{n-1}.\\
\end{aligned}
\right.
\end{equation*}
Then for any $0<\epsilon<\frac{1}{100}$, there exist constants
$$C_\epsilon>0\ \ \&\ \ 0<\delta=\delta(\epsilon)\ll\epsilon
$$
such that if:
\begin{itemize}
\item[\rm (i)] $Y=\cup_{k=1}^M B_k$ is a union of lattice $K^2$-cubes in $B^{n+1}(0,R)$ and each lattice $R^{\frac{1}{2}}$-cube intersecting $Y$ contains $\thicksim \lambda$ many $K^2$-cubes in $Y$, where $K=R^\delta$;

\item[\rm (ii)] $\|e^{it(-\Delta)^\alpha}f\|_{L^p(B_k)}$ is dyadically a constant in $k=1,2,\cdot\cdot\cdot,M$;

\item[\rm (iii)] $1\leq\kappa\leq n+1$ and $\gamma$ is given by
\begin{align}\label{2.1a1}
\gamma= \max_{ \substack {B^{n+1}(x',r)\subset B^{n+1}(0,R)\\
x'\in \mathbb{R}^{n+1}, r\geq K^2}}\frac{\#\{ B_k:B_k\subset B^{n+1}(x',r)\}}{r^\kappa},
\end{align}
\end{itemize}
then
\begin{align}\label{2.1a2}
\left\|e^{it(-\Delta)^\alpha}f\right\|_{L^p(Y)}
\leq C_\epsilon M^{-\frac{1}{n+1}}\gamma^{\frac{2}{(n+1)(n+2)}}\lambda^{\frac{n}{(n+1)(n+2)}}R^{\frac{\kappa}{(n+1)(n+2)}+\epsilon}\|f\|_{L^2(\mathbb{R}^n)}.
\end{align}
\end{theorem}

From Theorem \ref{theorem 2.1}, we can get the following $L^2$-restriction estimate.
\begin{corollary}\label{corollary 2.2}
	
	Let
	\begin{equation*}
	(n, R)\in\mathbb N \times[1,\infty)\ \ \&\ \
	\text{supp}\hat{f}\subset \mathbb B^n.
	\end{equation*}
	Then for any $\epsilon>0$ there exists a constant $C_\epsilon>0$ such that if:
	\begin{itemize}
		\item[\rm (i)] $X=\cup_k B_k$ is a union of lattice unit cubes in $B^{n+1}(0,R)$;
		
		\item[\rm (ii)] $1\leq\kappa\leq n+1$ and $\gamma$ is given by
		\begin{align}\label{2.2a1}
		\gamma=\max_{\substack{B^{n+1}(x',r)\subset B^{n+1}(0,R)\\x'\in \mathbb{R}^{n+1}, r\geq1}}\frac{\#\{ B_k:B_k\subset B^{n+1}(x',r)\}}{r^\kappa},
		\end{align}
		then
		\begin{align}\label{2.3a2}
		\left\|e^{it(-\Delta)^\alpha}f\right\|_{L^2(X)}
		\leq C_\epsilon \gamma^{\frac{1}{n+1}}R^{\frac{\kappa}{2(n+1)}+\epsilon}\|f\|_{L^2(\mathbb{R}^n)}.
		\end{align}
	\end{itemize}
\end{corollary}

\begin{proof}
For any $1\leq\lambda\leq R^{O(1)}$, we introduce the notation
$\mathcal{Z}_{\lambda}=\{B_k: B_k\subset X ~\text{such}~\text{that}~\text{any}~R^{\frac12}-~\text{cube}~\text{contains}~\sim \lambda~\text{unit}~\text{cubes}~B_k~\text{in}~\text{it}\}.$
By pigeonholing, we fix  $\lambda$ such that
$$\|e^{it(-\Delta)^\alpha}f\|_{L^2(X)}\lessapprox\|e^{it(-\Delta)^\alpha}f\|_{L^2(\cup_{B_k\in\mathcal{Z}_\lambda} B_k)}.$$
It is easy to see that
$$\lambda\leq\gamma R^{\frac{\kappa}{2}}$$
by taking
$r=R^{\frac{1}{2}}$
in \eqref{2.2a1}.

Next we assume the following inequality holds and we will prove this inequality later.
\begin{align}\label{2.2a2}
\|e^{it(-\Delta)^\alpha}f\|_{L^2(\cup_{B_k \in\mathcal{Z}_\lambda} B_k)}
\lessapprox \gamma^{\frac{2}{(n+1)(n+2)}}\lambda^{\frac{n}{(n+1)(n+2)}}R^{\frac{\kappa}{(n+1)(n+2)}}\|f\|_{L^2(\mathbb{R}^n)},
\end{align}
thereby reaching
$$\left\|e^{it(-\Delta)^\alpha}f\right\|_{L^2(X)}
\leq C_\epsilon \gamma^{\frac{1}{n+1}}R^{\frac{\kappa}{2(n+1)}+\epsilon}\|f\|_{L^2(\mathbb{R}^n)}.$$
Hence it remains to prove \eqref{2.2a2}. Denote $Z=\cup_{B_k \in\mathcal{Z}_\lambda} B_k$.

We can sort them into at most $O(\log R)$ many subsets of $Z$ according to the value of $\|e^{it(-\Delta)^\alpha}f\|_{L^p(B_k)}$. In each subset the value of $\|e^{it(-\Delta)^\alpha}f\|_{L^p(B_k)}$ is dyadically a constant. Among the subsets we can find a set $Z'\subset Z$ such that
$$\{\|e^{it(-\Delta)^\alpha}f\|_{L^p(B_k)}: B_k\subset Z'\} \text{ is dyadically a constant}$$ and
$$\|e^{it(-\Delta)^\alpha}f\|_{L^2(Z)}\lessapprox\|e^{it(-\Delta)^\alpha}f\|_{L^2(Z')}.$$
Upon writing
$$M=\#\{B: B~\textrm{is}~\textrm{unit}~\textrm{cube}~\textrm{and}~B\subset Z'\},$$
and using H\"{o}lder's inequality, we have
$$
\left\|e^{it(-\Delta)^\alpha}f\right\|_{L^2(Z)}
\lessapprox\left\|e^{it(-\Delta)^\alpha}f\right\|_{L^2(Z')}
\leq\left\|e^{it(-\Delta)^\alpha}f\right\|_{L^p(Z')}|Z'|^{\frac{1}{2}-\frac{1}{p}}
\leq M^{\frac{1}{n+1}}\left\|e^{it(-\Delta)^\alpha}f\right\|_{L^p(Z')}.
$$
So, in order to prove \eqref{2.2a2}, it suffices to prove

\begin{align}\label{2.2b2}
\left\|e^{it(-\Delta)^\alpha}f\right\|_{L^p(Z')}
\lessapprox M^{-\frac{1}{n+1}}\gamma^{\frac{2}{(n+1)(n+2)}}\lambda^{\frac{n}{(n+1)(n+2)}}R^{\frac{\kappa}{(n+1)(n+2)}}\|f\|_{L^2(\mathbb{R}^n)}.
\end{align}
In order to use the result of Theorem \ref{theorem 2.1}, we need to extend the size of the unit cube to $K^2$-cube according to the following two steps.
\begin{itemize}

\item[Step 1.] Let $\beta$ be a dyadic number and
$\mathcal{B}_\beta:=\{B:  B\subset Z' ~\textrm{and}~\textrm{for}~\textrm{any}~\textrm{the}~\textrm{lattice}~ K^2-\textrm{cube}~ \tilde{B}\supset B~\textrm{such}~\textrm{that}~ \|e^{it(-\Delta)^\alpha}f\|_{L^p(\tilde{B})}\thicksim\beta\} $,
and set $$\mathcal{\tilde{B}}_\beta=\{\tilde{B}:\textrm{the}~\textrm{relevant}~K^2-\textrm{cubes}\}.$$
\item[Step 2.] Next, fixing $\beta$, letting $\lambda'$ be a dyadic number, and denoting
 $$
 \begin{cases}\mathcal{B}_{\beta,\lambda'}=\{B\in \mathcal{B}_{\beta}: R^{\frac{1}{2}}-\text{cube  } Q~\text{contains}~\lambda'~\text{many}~K^2-\text{cubes}~\text{from}~\tilde{\mathcal{B}}_\beta\};\\
\mathcal{\tilde{B}}_{\beta,\lambda'}=\{\tilde{B}:\textrm{the}~\textrm{relevant}~K^2-\textrm{cubes}\},
\end{cases}
$$
we find that
the pair
$\{\beta,\lambda'\}$ satisfies
$$M'=\#\mathcal{\tilde{B}}_{\beta,\lambda'}\gtrapprox M.$$
From the definition of $\lambda$ and $\gamma$, we have

\begin{equation*}
\left\{
\begin{aligned}
 &\lambda'\leq\lambda;\\
&\gamma'=\max_{\substack{B^{n+1}(x',r)\subset B^{n+1}(0,R)\\x'\in \mathbb{R}^{n+1}, r\geq K^2}}
\frac{\#\{\tilde{B}:\tilde{B}\in\mathcal{\tilde{B}}_{\beta,\lambda'}, \tilde{B}\subset B^{n+1}(x',r)\}}{r^\kappa}
\leq\gamma.\\
\end{aligned}
\right.
\end{equation*}
If $$Y=\cup_{\tilde{B}\in\mathcal{\tilde{B}}_{\beta,\lambda'}},
$$
then Theorem \ref{theorem 2.1} yields
\begin{align*}
\left\|e^{it(-\Delta)^\alpha}f\right\|_{L^p(Z')}
&\lessapprox\left\|e^{it(-\Delta)^\alpha}f\right\|_{L^p(Y)}\\
&\lessapprox M'^{-\frac{1}{n+1}}\gamma'^{\frac{2}{(n+1)(n+2)}}\lambda'^{\frac{n}{(n+1)(n+2)}}R^{\frac{\kappa}{(n+1)(n+2)}}\|f\|_{L^2(\mathbb{R}^n)}\\
&\lessapprox M^{-\frac{1}{n+1}}\gamma^{\frac{2}{(n+1)(n+2)}}\lambda^{\frac{n}{(n+1)(n+2)}}R^{\frac{\kappa}{(n+1)(n+2)}}\|f\|_{L^2(\mathbb{R}^n)},
\end{align*}
which is the desired \eqref{2.2b2}.
\end{itemize}


\end{proof}

\subsection{Proof of Theorem \ref{theorem 1.5} }
In this section,  we use Corollary \ref{corollary 2.2} to prove Theorem \ref{theorem 1.5}.
\begin{proof}[Proof of (Corollary \ref{corollary 2.2} $\Rightarrow$Theorem \ref{theorem 1.5})] This proceeds below.

\begin{itemize}	
\item[$\rhd$]
We have
$$\text{supp}\hat{f}\subset \mathbb B^n\Rightarrow\text{supp}~(e^{it(-\Delta)^\alpha}f)^\wedge  \subset \mathbb B^{n+1}.
$$
Thus,
\begin{equation*}
\exists\ \
 \psi\in \mathcal S(\mathbb{R}^{n+1})\ \ \&\ \ \hat{\psi}\equiv1~\textrm{on}~B^{n+1}(0,2)
\ \ \text{
such that}\ \
(e^{it(-\Delta)^\alpha}f)^2=(e^{it(-\Delta)^\alpha}f)^2\ast\psi.
\end{equation*}

\item[$\rhd$] If $$\max_{|\tilde{y}-(x,t)|\leq e^{5}}|\psi(\tilde{y})|=\psi_1(x,t)$$
which decays rapidly, then for any $(x,t)\in\mathbb{R}^{n+1}$, $$\tilde{m}(x,t)=(m,m_{n+1})=(m_1,\cdots,m_n,m_{n+1})$$ denotes the center of the unit lattice cube containing $(x,t)$, and hence
$$\Big(|e^{it(-\Delta)^\alpha}f|^2\ast|\psi|\Big)(x,t)\leq \Big(|e^{it(-\Delta)^\alpha}f|^2\ast\psi_1\Big)(\tilde{m}(x,t)).$$
Accordingly,
\begin{align}\label{1.5b1}
&\left\|\sup_{0<t<R}\left|e^{it(-\Delta)^\alpha}f\right|\right\|^2_{L^2(B^n(0,R);\mu_R)}\\
&\ \ =\int_{B^n(0,R)} \sup_{0<t<R}\left|e^{it(-\Delta)^\alpha}f(x)\right|^2 d\mu_R(x)\nonumber\\
&\ \ \leq\int_{B^n(0,R)} \sup_{0<t<R} \left(|e^{it(-\Delta)^\alpha}f|^2\ast|\psi|\right)(x,t) d\mu_R(x)\nonumber\\
&\ \ \leq\int_{B^n(0,R)} \sup_{0<t<R} \left(|e^{it(-\Delta)^\alpha}f|^2\ast\psi_1 \right)(\tilde{m}(x,t)) d\mu_R(x)\nonumber\\
&\ \ \leq \sum_{\substack{m=(m_1,\cdot\cdot\cdot,m_n)\in\mathbb{Z}^n\\|m_i|,|m_{n+1}|\leq R}} \left(\int_{|x-m|\leq 10}d\mu_R(x)\right)
\cdot\sup_{\substack{m\in\mathbb{Z}^n\\0\leq m_{n+1}\leq R}} \left(|e^{it(-\Delta)^\alpha}f|^2\ast\psi_1 \right)(m,m_{n+1}).\nonumber
\end{align}

\item[$\rhd$] For each $m\in \mathbb{Z}^n$, let  $b(m)$ be an integer in $[0,R]$ such that
$$\sup_{\substack{m_{n+1}\in\mathbb{Z}\\0\leq m_{n+1}\leq R}} \left(|e^{it(-\Delta)^\alpha}f|^2\ast\psi_1 \right)(m,m_{n+1})=\left(|e^{it(-\Delta)^\alpha}f|^2\ast\psi_1 \right)(m,b(m)).$$
Next, via defining
$$v_m=\int_{|x-m|\leq 10}d\mu_R(x)\lesssim 1,
$$
and using \eqref{1.5b1}, we have
\begin{align}\label{1.5b2}
\left\|\sup_{0<t<R}\left|e^{it(-\Delta)^\alpha}f\right|\right\|^2_{L^2(B^n(0,R);\mu_R)}
\lesssim\sum_{\substack{v ~\textrm{dyadic}\\ v\in[R^{-10n},1]}} \sum_{\substack{m\in\mathbb{Z}^n,|m_i|\leq R\\v_m\thicksim v}} v
\cdot\left(|e^{it(-\Delta)^\alpha}f|^2\ast\psi_1 \right)(m,b(m))+R^{-9n}.
\end{align}
By pigeonholing, we get that for any small $\epsilon>0$,
\begin{align}\label{1.5b3}
\left\|\sup_{0<t<R}\left|e^{it(-\Delta)^\alpha}f\right|\right\|^2_{L^2(B^n(0,R);\mu_R)}
&\lessapprox\sum_{\substack{m\in\mathbb{Z}^n,|m_i|\leq R\\v_m\thicksim v}} v\cdot\left(|e^{it(-\Delta)^\alpha}f|^2\ast\psi_1 \right)(m,b(m))+R^{-8n}\\
&\lesssim\sum_{\substack{m\in\mathbb{Z}^n,|m_i|\leq R\\v_m\thicksim v}} v\cdot \left(\int_{B^{n+1}((m,b(m)),R^\epsilon)} |e^{it(-\Delta)^\alpha}f|^2 \right)+R^{-8n}\nonumber\\
&\lesssim v \cdot \int_{\cup_{m\in A_v}B^{n+1}((m,b(m)),R^\epsilon)} |e^{it(-\Delta)^\alpha}f|^2 +R^{-8n}.\nonumber
\end{align}

\item[$\rhd$] Note that
$$X_v=\cup_{m\in\mathbb{Z}^n: |m_i|\leq R~\textrm{and}~v_m\thicksim v} B^{n+1}((m,b(m)),R^\epsilon)
$$
is not only a union of some distinct $R^\epsilon$-balls but also a union of some unit balls. So, these balls' projections onto the $(x_1,\cdot\cdot\cdot,x_n)$-plane are essentially disjoint (a point can be covered $\lesssim R^\epsilon$ times). For every $r>R^{2\epsilon}$, the definition of $\{m\in\mathbb{Z}^n: |m_i|\leq R~\textrm{and}~v_m\thicksim v\}$ ensures that the intersection of $X_v$ and any $r$-ball can be contained in  $\lesssim R^{10n\epsilon}v^{-1}r^\kappa$ disjoint $R^\epsilon$- balls. Hence we can apply Corollary \ref{corollary 2.2} to $X_v$ with
$$\gamma\lesssim R^{100n\epsilon}v^{-1} \ \ \& \ \ 1\leq\kappa\leq n+1.$$
By \eqref{1.5b3}, we reach \eqref{1.5a} via
\begin{equation*}\label{1.5b4}
\left\|\sup_{0<t<R}\left|e^{it(-\Delta)^\alpha}f\right|\right\|^2_{L^2(B^n(0,R);\mu_R)}
\lesssim v\left(\gamma^{\frac{1}{n+1}}R^{\frac{\kappa}{2(n+1)}+\epsilon}\|f\|_{L^2(\mathbb{R}^n)}\right)^2\lessapprox v^{\frac{n-1}{n+1}}R^{\frac{\kappa}{n+1}} \|f\|^2_{L^2(\mathbb{R}^n)}\lesssim R^{\frac{\kappa}{n+1}} \|f\|^2_{L^2(\mathbb{R}^n)}.
\end{equation*}

\end{itemize}

\end{proof}

\section{Conclusion}\label{s4}
\subsection{Proof of Theorem \ref{theorem 2.1} -  $R\lesssim 1$}\label{s41}

In what follows, we always assume
$$
\begin{cases}
p=\frac{2(n+1)}{n-1};\\
q=\frac{2(n+2)}{n};\\
\text{supp}\hat{f}\subset \mathbb B^n.
\end{cases}
$$
But nevertheless, the estimate  \eqref{2.1a2} under $R\lesssim 1$ is trivial. In fact,
from the assumptions of Theorem \ref{theorem 2.1}, we see
$$M\sim\lambda\sim\gamma\sim R\sim1.$$
Furthermore, by the short-time Strichartz estimate (see \cite{COX,Dinh}), we get
\begin{equation}
\left\|e^{it(-\Delta)^\alpha}f\right\|_{L^p(Y)}
\leq\left\|e^{it(-\Delta)^\alpha}f\right\|_{L^p([0,1]\times\mathbb{R}^n)}\lesssim \left\|f\right\|_{L^2(\mathbb{R}^n)},
\end{equation}
thereby verifying Theorem \ref{theorem 2.1} for $R\lesssim1$.

\subsection{Proof of Theorem \ref{theorem 2.1} - $R\gg 1$}\label{s42} This goes below.

1stly, we decompose the unit ball in the frequency space into disjoint $K^{-1}$-cubes $\tau$. Write

$$
\begin{cases}\mathcal{S}=\big\{\tau: K^{-1}-\text{cubes}~ \tau\subset \mathbb B^n\big\};\\
f=\sum_\tau f_\tau;\\ \widehat{f_\tau}=\hat{f}\chi_\tau;\\
\mathcal{S}(B)
=\left\{\tau\in \mathcal{S}: \left\|e^{it(-\Delta)^\alpha}f_\tau \right\|_{L^p(B)}\geq \frac{1}{100(\#\mathcal{S})}
\left\|e^{it(-\Delta)^\alpha}f \right\|_{L^p(B)}\right\}\ \ \text{for a}\ K^2-\text{cube}\ B.
\end{cases}
$$
Then
$$\left\|\sum_{\tau\in \mathcal{S}(B)}e^{it(-\Delta)^\alpha}f_\tau \right\|_{L^p(B)}
\thicksim\left\|e^{it(-\Delta)^\alpha}f \right\|_{L^p(B)}.$$

2ndly, we recall the definitions of narrow cube and broad cube.
\begin{itemize}
\item[$\rhd$]
We say that a $K^2$-cube $B$ is narrow if there is an $n$-dimensional subspace $V$ such that for all $\tau\in \mathcal{S}(B)$
$$ \angle(G(\tau),V)\leq\frac{1}{100nK},$$
where $G(\tau)\subset\mathbb{S}^n$ is a spherical cap of radius $\thicksim K^{-1}$ given by
$$G(\tau)=\left\{\frac{(-2\xi,1)}{|(-2\xi,1)|}\in \mathbb{S}^n: \xi\in\tau\right\},$$
and $\angle(G(\tau),V)$ denotes the smallest angle between any non-zero vector $v\in V$ and $v'\in G(\tau)$.

\item [$\rhd$]
Otherwise we say that the $K^2$-cube $B$ is broad. In other words, a cube being broad means that the tiles $\tau\in \mathcal{S}(B)$ are so separated such that the norm vectors of the corresponding spherical caps can not be in an $n$-dimensional subspace - more precisely - for any broad $B$,

\begin{align}\label{2.1b1}
\exists\ \tau_1,\cdot\cdot\cdot,\tau_{n+1}\in \mathcal{S}(B)\ \ \text{such that}\ \
|v_1\wedge v_2\wedge\cdot\cdot\cdot\wedge v_{n+1}|\gtrsim K^{-n}\ \ \forall\ \ v_j\in G(\tau_j).
\end{align}
\end{itemize}

3rdly, with the setting:
\begin{equation*}
\left\{
\begin{aligned}
 &Y_{\textrm{broad}}=\cup_{B_k~\textrm{is}~\textrm{broad}}B_k;\\
&Y_{\textrm{narrow}}=\cup_{B_k~\textrm{is}~\textrm{narrow}}B_k,\\
\end{aligned}
\right.
\end{equation*}
we will handle $Y$ according to the sizes of $Y_{\textrm{broad}}$ and $Y_{\textrm{narrow}}$.
\begin{itemize}
\item[(1)]
We call it the broad case if $Y_{\textrm{broad}}$ contains $\geq\frac{M}{2}$ many $K^2$-cubes and we will deal with the broad case using the multilinear refined Strichartz estimates.

\item[(2)]
We call it the narrow case if $Y_{\textrm{narrow}}$ contains $\geq\frac{M}{2}$ many $K^2$-cubes and we will handle the narrow case by $l^2$-decoupling, parabolic rescaling and induction on scales.
\end{itemize}

\subsubsection{The broad case.}
In this case, we consider the same generalized Schr\"{o}dinger operators as  Cho-Ko \cite{CK}. The idea here is to take it as a close perturbation of the typical curve $|\xi|^2$ in very small scale and keep this perturbation under parabolic scaling. This can not be true for $|\xi|^{2\alpha}$ with $\alpha>\frac 12$. But it is true for its quadratic term. This is the reason to introduce  the following set $\mathcal{NPF}(L,c_0)$ and apply induction in this set. Let us recall the two definitions in \cite{CK}.
\begin{itemize}
 \item[$\rhd$]  Let $\Phi(D)$ be a multiplier operator defined on $\mathbb{R}^n$ which satisfies:
\begin{align}\label{smooth}
\begin{cases}
\Phi(\xi) ~\text{is smooth at}\ \xi\neq0;\\
|D^\beta\Phi(\xi)|\lesssim |\xi|^{2\alpha-|\beta|}\ \&\ |\nabla\Phi(\xi)|\gtrsim|\xi|^{2\alpha-1}\ \forall\ \text{multi-index}\ \beta;\\
\textrm{The Hessian matrix of}~\Phi~\text{is positive definite}.
\end{cases}
\end{align}

\item[$\rhd$] Let $0<c_0\ll1$ and $L\in\mathbb{N}$ be sufficiently large. We consider a collection of the normalized phase functions:
$$\mathcal{NPF}(L,c_0)=\left\{\Phi\in C_0^\infty(B^n(0,2)):\left \|\Phi(\xi)-\frac{|\xi|^2}{2}\right\|_{C^L(\mathbb B^n)}\leq  c_0\right\}.$$
\end{itemize}

\begin{theorem}\label{theorem 3.1}(Linear refined Strichartz estimate in dimension $n+1$).
Suppose that
\begin{itemize}
	\item[\rm (i)] $\Phi$ is in $\mathcal{NPF}(L,c_0)$ for sufficiently small $c_0>0$;
	
	\item[\rm (ii)] $\{Q_j\}$ is a sequence of the lattice $R^{\frac{1}{2}}$-cubes in $B^{n+1}(0,R)$ with
$\|e^{it\Phi }f\|_{L^q(Q_j)}$ being essentially constant in $j$;

\item[\rm (iii)] $\{Q_j\}$ is arranged in horizontal slabs of the form $\mathbb{R}\times\cdot\cdot\cdot\times\mathbb{R}\times\{t_0,t_0+R^{\frac{1}{2}}\}$ which contains $\thicksim\sigma$ cubes $Q_j$.
\end{itemize}
Then
\begin{equation}\label{eq:2}\|e^{it\Phi}f\|_{L^q(\cup_j Q_j)}\leq C_\epsilon R^\epsilon \sigma^{-\frac{1}{n+2}}\|f\|_{L^2(\mathbb{R}^n)}\ \ \forall\ \ \epsilon>0.\end{equation}
\end{theorem}

\begin{remark}\label{r41} On the one hand, by taking $\Phi(\xi)=|\xi|^2$, we can rediscover the results for the Schr\"{o}dinger operator by  Du-Guth-Li \cite{DGL} in $\mathbb{R}^{2+1}$ and \cite{DZ} in higher dimensional cases. Similar results can also be found in \cite{CK} with an extral restriction condition on the support of $f$.

On the other hand, for $\Phi(\xi)=|\xi|^{2\alpha}$ with $\alpha>\frac12$ we can reduce $\Phi$ satisfying \eqref{smooth} to a function in $\mathcal{NPF}(L,c_0)$. Denote by $H\Phi(\xi_0)$ the Hessian matrix of $\Phi(\xi)$ at point $\xi_0$. Since the Hessian matrix  of $\Phi$ is positive definite, we can write it as $H\Phi(\xi_0)=P^{-1}DP$ with $P$ a symmetric matrix $D=(\lambda_1\textbf{e}_1,\cdot\cdot\cdot,\lambda_n\textbf{e}_n)$ and $\lambda_1>0,\cdot\cdot\cdot,\lambda_n>0$. We introduce a new function around point $\xi_0$:
\begin{align}\label{NPF}
\Phi_{\rho,\xi_0}(\xi)=\rho^{-2}\left(\Phi(\rho H^{-1}\xi+\xi_0)-\Phi(\xi_0)-\rho\nabla\Phi(\xi_0)\cdot H^{-1}\xi\right),
\end{align}
 From Cho-Ko \cite{CK}, we have $\Phi_{\rho,\xi_0}\in\mathcal{NPF}(L,c_0)$ for a sufficiently small  $\rho=\rho(\Phi,L,c_0)>0$. Moreover
\begin{align*}
\left|e^{it\Phi}f(x)\right|
&=(2\pi)^{-n}\left|\int_{\mathbb{R}^n}e^{i(x,t)\cdot(\xi,\Phi(\xi))}\hat{f}(\xi)d\xi\right|\\
&=(2\pi)^{-n}\left|\int_{\mathbb{R}^n}e^{i(x,t)\cdot(\rho H^{-1}\eta+\xi_0,\Phi(\rho H^{-1}\eta+\xi_0))}\hat{f}(\rho H^{-1}\eta+\xi_0)\rho^n |H|^{-1}d\eta\right|\\
&=\rho^n |H|^{-1}(2\pi)^{-n}\left|\int_{\mathbb{R}^n}e^{i(\rho H^{-t}x+\rho tH^{-t}\nabla\Phi(\xi_0),\rho^2t)\cdot(\eta,\Phi_{\rho,\xi_0}(\eta))}\hat{f}(\rho H^{-1}\eta+\xi_0)d\eta\right|.
\end{align*}
Next, we use
\begin{equation*}
\left\{
\begin{aligned}
 &x'=\rho H^{-t}( x+t\nabla\Phi(\xi_0));\\
&t'=\rho^2t;\\
&\widehat{f}_{\rho,\xi_0}(\eta)=\rho^{\frac{n}{2}}|H|^{-\frac{1}{2}}\hat{f}(\rho H^{-1}\eta+\xi_0);\\  &\|f\|_{L^2(\mathbb{R}^n)}=\|f_{\rho,\xi_0}\|_{L^2(\mathbb{R}^n)},
\end{aligned}
\right.
\end{equation*}
to get
\begin{align*}
\left\|e^{it\Phi}f\right\|^q_{L^q(S)}
&=\int_S|e^{it\Phi}f(x)|^qdxdt\\
&=\int_S\left|\rho^n |H|^{-1}(2\pi)^{-n}\int_{\mathbb{R}^n}e^{i(\rho H^{-t}x+\rho tH^{-t}\nabla\Phi(\xi_0),\rho^2t)\cdot(\eta,\Phi_{\rho,\xi_0}(\eta))}\hat{f}(\rho H^{-1}\eta+\xi_0)d\eta\right|^qdxdt\\
&=\rho^{nq} |H|^{-q}\int_{S'}\left|(2\pi)^{-n}\int_{\mathbb{R}^n}e^{i(x',t')\cdot(\eta,\Phi_{\rho,\xi_0}(\eta))}\hat{f}(\rho H^{-1}\eta+\xi_0)d\eta\right|^q\rho ^{-n}|H|dx'\rho^{-2}dt'\\
&=\rho^{nq-n-2-\frac{nq}{2}} |H|^{-q+1+\frac{q}{2}}\int_{S'}\left|(2\pi)^{-n}\int_{\mathbb{R}^n}e^{i(x',t')\cdot(\eta,\Phi_{\rho,\xi_0}(\eta))}
\rho^{\frac{n}{2}}|H|^{-\frac{1}{2}}\hat{f}(\rho H^{-1}\eta+\xi_0)d\eta\right|^qdx'dt'\\
&=\rho^{\frac{nq}{2}-n-2} |H|^{-\frac{q}{2}+1}\int_{S'}\left|(2\pi)^{-n}\int_{\mathbb{R}^n}e^{i(x',t')\cdot(\eta,\Phi_{\rho,\xi_0}(\eta))}
\hat{f}_{\rho,\xi_0}(\eta)d\eta\right|^qdx'dt'\\
&=\rho^{\frac{nq}{2}-n-2} |H|^{-\frac{q}{2}+1}\left\|e^{it'\Phi_{\rho,\xi_0}}f_{\rho,\xi_0}\right\|^q_{L^q(S')}.
\end{align*}
In short, we have
\begin{align}\label{3.1b1}
\left\|e^{it\Phi}f\right\|_{L^q(S)}
=\rho^{\frac{n}{2}-\frac{n+2}{q}} |H|^{\frac{1}{q}-\frac{1}{2}}\left\|e^{it'\Phi_{\rho,\xi_0}}f_{\rho,\xi_0}\right\|_{L^q(S')}.
\end{align}
Note that
$$\frac{n}{2}-\frac{n+2}{q}=0 \ \ \& \ \  |H|\sim 1~ ( \text{since}~ \text{supp}\hat f\subset \{\xi: |\xi|\thicksim 1\})
$$
and the change of variables does not change the value of $\sigma$. So \eqref{eq:2} is also true for the generalized phase functions $\Phi$ satisfying \eqref{smooth} which contains $\Phi(\xi)=|\xi|^{2\alpha}$ with $\alpha>\frac12$.

\end{remark}
\begin{lemma}\label{Decoupling}(Bourgain-Demeter's $l^2$-decoupling inequality \cite{BD}).
Suppose that $\hat{g}$ is supported in a $\sigma$-neighborhood of an elliptic surface $S$ in $\mathbb{R}^n$. If $\tau$ is a rectangle of size $\sigma^{\frac{1}{2}}\times\cdot\cdot\cdot\times\sigma^{\frac{1}{2}}\times\sigma$ inside $\sigma$-neighborhood of $S$, $\widehat{g_\tau}=\hat{g}\chi_\tau$ and $\epsilon>0$, then
$$\left\|g\right\|_{L^p(\mathbb{R}^n)}
\leq C_\epsilon \sigma^{-\epsilon}
\left(\sum_\tau\left\|g_\tau\right\|^2_{L^p(\mathbb{R}^n)}\right)^{\frac{1}{2}}.$$
\end{lemma}

\begin{proof}[Proof of Theorem \ref{theorem 3.1}]
Now we prove linear refined Strichartz estimate in dimension $n+1$ by four steps.

\begin{itemize}

\item[$\rhd$]
1stly, we consider the wave packet decomposition of $f$. For any smooth function $f: \mathbb B^n\rightarrow \mathbb R$, we decompose it into wave packets and each wave packet supported in a ball $\theta$ of radius $R^{-\frac14}$. Then we divide the physical space $B^n(0,R)$ into balls $D$ of radius $R^{\frac34}$. From \cite{T}, we have
$$f=\sum_{\theta, D} f_{T_{\theta, D}}\ \  \& \ \ f_{T_{\theta, D}}=(\hat f \chi_{\theta})^{\vee}\chi_D.$$
And we have the functions $f_{T_{\theta, D}}$ are approximately orthogonal, thereby getting
$$\|f\|^2_{L^2(\mathbb R^n)}\sim \sum _{\theta, D} \|f_{T_{\theta, D}}\|^2_{L^2(\mathbb R^n)}.$$
By computation, we have the restriction of $e^{it\Phi}f_{T_{\theta, D}}(x)$ to $B^{n+1}(0,R)$ is essentially supported on a tube $T_{\theta, D}$
which is defined as follows:
$$T_{\theta, D}=\left\{(x,t):(x,t)\in B^{n+1}(0,R) \ \ \& \ \ \left|x-c(D)-t\nabla\Phi(c(\theta))\right|\leq R^{\frac{3}{4}+\delta} \ \ \& \ \ 0<t<R\right\}.$$
Here $c(\theta)\ \&\ c(D)$ denote the centers of $\theta\ \&\ D$ respectively.
Therefore, by decoupling theorem, we have
$$\left\|e^{it\Phi}f\right\|_{L^q(Q)}
\lessapprox\left(\sum_T\left\|e^{it\Phi}f_T\right\|^2_{L^q(Q)}\right)^{\frac{1}{2}},$$
where $T_{\theta, D}=T$.
In fact, we take $\eta_Q\in \mathcal S(\mathbb{R}^{n+1})$ such that $\text{supp}~\widehat{\eta_Q}\subset Q^*$ and $Q^*$ is $R^{-\frac{1}{2}}$-cube. And we have $|\eta_Q|\sim1$ on $Q$.
By Lemma \ref{Decoupling}, we obtain
\begin{align*}
\left\|e^{it\Phi}f\right\|_{L^q(Q)}
\lesssim \left\|e^{it\Phi}f \eta_Q\right\|_{L^q(\mathbb{R}^{n+1})}\lesssim \left(\sum_T\left\|e^{it\Phi}f_T \eta_Q\right\|^2_{L^q(\mathbb{R}^{n+1})}\right)^{\frac{1}{2}}\lesssim \left(\sum_T\left\|e^{it\Phi}f_T\right\|^2_{L^q(Q)}\right)^{\frac{1}{2}}.
\end{align*}

\end{itemize}
\begin{itemize}

\item[$\rhd$]

2ndly, we use parabolic rescaling and induction on radius $R^{\frac12}$.
It goes as follows:

Suppose that:
\begin{itemize}
	\item $\{S_j\}_j $ are $R^{\frac{1}{2}}\times\cdot\cdot\cdot\times R^{\frac{1}{2}}\times R^{\frac{3}{4}}$-tubes in $T$ which is parallel to the long axes of $T$;
	\item
$\|e^{it\Phi}f_T\|_{L^q(S_j)}$ is essentially dyadically constant in $j$;

\item these tubes are arranged into $R^{\frac{3}{4}}$-slabs running parallel to the short axes of $T$ which contains $\thicksim\sigma_T$ tubes $S_j$;


\item $Y_T=\cup_j S_j.$
\end{itemize}
Then
\begin{align}\label{3.1b2}
\left\|e^{it\Phi}f_T\right\|_{L^q(Y_T)}
\leq C_\epsilon R^{\frac{\epsilon}{2}}\sigma_T^{-\frac{1}{n+2}}\|f_T\|_{L^2(\mathbb{R}^n)}.
\end{align}
In fact, 
as in Remark \ref {r41}, we get
\begin{align}\label{3.1b3}
\begin{cases}
\left\|e^{it\Phi}f\right\|_{L^q(S)}
=\rho^{\frac{n}{2}-\frac{n+2}{q}} |H|^{\frac{1}{q}-\frac{1}{2}}\left\|e^{it'\Phi_{\rho,\xi_0}}f_{\rho,\xi_0}\right\|_{L^q(S')};\\
\widehat{f}_{\rho,\xi_0}(\eta)=\rho^{\frac{n}{2}}|H|^{-\frac{1}{2}}\hat{f}(\rho H^{-1}\eta+\xi_0);\\ \|f\|_{L^2(\mathbb{R}^n)}=\|f_{\rho,\xi_0}\|_{L^2(\mathbb{R}^n)}.
\end{cases}
\end{align}
If
$$\rho=R^{-\frac{1}{4}}  \ \ \& \ \ \xi_0=c(D) \ \ \& \ \ S=Y_T \ \ \& \ \ S'=\widetilde{Y},$$
then $\widetilde{Y}$, as the image of $Y_T$ under the new coordinate, is a union of $R^{\frac{1}{4}}$-cubes inside an $R^{\frac{1}{2}}$-cube. These $R^{\frac{1}{4}}$-cubes are arranged in $R^{\frac{1}{4}}$-horizontal slabs, and
$$\#\{R^{\frac{1}{4}}-\text{cubes}: R^{\frac{1}{4}}-\text{cubes }~\text{are}~\text{arranged }~\text{in}~R^{\frac{1}{4}}-\text{horizontal}~\text{slabs}\}\sim \sigma_T,$$
and hence
$$
\left\|e^{it\Phi}f\right\|_{L^q(Y_T)}
=|H|^{-\frac{1}{n+2}}\left\|e^{it'\Phi_{\rho,\xi_0}}f_{\rho,\xi_0}\right\|_{L^q(\widetilde{Y})}.$$
From induction we have
$$\left\|e^{it'\Phi_{\rho,\xi_0}}f_{\rho,\xi_0}\right\|_{L^q(\widetilde{Y})}
\leq C_\epsilon R^{\frac{\epsilon}{2}}\sigma_T^{-\frac{1}{n+2}}\|f_{\rho,\xi_0}\|_{L^2(\mathbb{R}^n)},
$$
thereby getting that if $f=f_T$ then
$$\left\|e^{it\Phi}f_T\right\|_{L^q(Y_T)}
\leq C_\epsilon|H|^{-\frac{1}{n+2}} R^{\frac{\epsilon}{2}}\sigma_T^{-\frac{1}{n+2}}\|f_T\|_{L^2(\mathbb{R}^n)}
\lesssim R^{\frac{\epsilon}{2}}\sigma_T^{-\frac{1}{n+2}}\|f_T\|_{L^2(\mathbb{R}^n)},~\left(\text{thanks to}~ |H|\thicksim 1\right)$$
namely, \eqref{3.1b2} holds.
\end{itemize}

\begin{itemize}

\item[$\rhd$]
3rdly, we shall choose an appropriate $Y_T$. For each $T$, we classify tubes in $T$ in the following ways.
\begin{itemize}

\item  For each dyadic number $\lambda$, we define
  $\mathbb{S}_\lambda=\left\{S_j: S_j\subset T \ \ \& \ \ \left\|e^{it\Phi}f_T\right\|_{L^q(S_j)}\thicksim \lambda\right\}$.
  \item For any dyadic number $\eta$, we define
  $\mathbb{S}_{\lambda,\eta}=\left\{S_j: S_j\in\mathbb{S}_\lambda\ \ \& \ \ \#\{S_j,S_j\subset R^{\frac{3}{4}}-\textrm{slab}\}\thicksim\eta\right\}$.
\end{itemize}
We denote
$$Y_{T,\lambda,\eta}=\cup_{S_j\in\mathbb{S}_{\lambda,\eta}}S_j,$$
thereby getting
$$e^{it\Phi}f=\sum_{\lambda,\eta}\left(\sum_{T}e^{it\Phi}f_T\cdot\chi_{Y_{T,\lambda,\eta}}\right).$$

For each $\lambda,\eta$, there are $O(\log R)$ choices. By pigeonholing, we can choose $\lambda,\eta$ so that
$$\left\|e^{it\Phi}f\right\|_{L^q(Q_j)}
\lesssim(\log R)^2\left\|\sum_{T}e^{it\Phi}f_T\cdot\chi_{Y_{T,\lambda,\eta}}\right\|_{L^q(Q_j)}$$
holds for $\approx 1$ of all cubes $Q_j\subset Y$, where  $Y=\cup_j Q_j$. In fact, we have $\#\{Q_j\}_j \lesssim R^\frac{n+1}{2}\ \ \& \ \ \#\{\lambda,\eta\} \lesssim \log R$. Since $\log R\ll R^\frac{n+1}{2}$, this inequality  holds for $\approx 1$ of all cubes $Q_j\subset Y$.
Here $(\lambda,\eta)$ is independent of $Q_j$.
\begin{itemize}
\item First of all, we fix $\lambda,\eta$ in the sequel of the proof of refined Strichartz estimate in dimension $n+1$. Let $Y_{T,\lambda,\eta}=Y_T$ for convenience. Note that $Y_T$ satisfies the hypotheses for our inductive estimate, where $\sigma_T=\eta$. By the definition of $Y_T \ \ \& \ \ \sigma_T$ and the direction of $T$, we have
$Y_T$ contains $\lesssim\sigma_T$ cubes $Q_j$ in any $R^{\frac{1}{2}}$-horizontal slab. Therefore,
\begin{align}\label{3.1c1}
\left|Y_T\cap Y\right|\lesssim\frac{\sigma_T}{\sigma}\left|Y\right|.
\end{align}

\item Next, we choose the tubes $Y$ according to the dyadic size of $\|f_T\|_{L^2(\mathbb{R}^n)}$. We can restrict matters to $O(\log R)$ choices of this dyadic size, and so we can choose a set of $T$'s, $\mathbb{T}$ such that $$\|f_T\|_{L^2(\mathbb{R}^n)} ~\text{is}~\text{essentially}~\text{constant}$$ and
\begin{align}\label{3.1c2}
\left\|e^{it\Phi}f\right\|_{L^q(Q_j)}
\lessapprox\left\|\sum_{T\in\mathbb{T}}e^{it\Phi}f_T\cdot\chi_{Y_{T}}\right\|_{L^q(Q_j)}~\text{holds}~\text{for}\approx 1~\text{of}~\text{all}~\text{cubes}~Q_j\subset Y.
\end{align}

\item Last of all, we choose the cubes $Q_j\subset Y$ according to the number of $Y_T$ that contain them. Denote by
$$Y'=\{Q_j: Q_j \subset Y ~\text{which}~\text{obey }~\eqref{3.1c2}~\text{and}~\text{each}~Q_j~\text{lie}~\text{in}~\sim \nu~\text{of}~\text{the}~\text{sets} ~\{Y_T\}_{T\in\mathbb{T}}\}.$$
 Because \eqref{3.1c2} holds for $\approx 1$ cubes and $\nu$ are dyadic numbers, we can use \eqref{3.1c1} to get
$$|Y'|\thickapprox|Y|\ \ \&\ \
|Y_T\cap Y'|\leq|Y_T\cap Y|\lesssim\frac{\sigma_T}{\sigma}|Y|\approx\frac{\sigma_T}{\sigma}|Y'|,$$
thereby finding
\begin{align}\label{3.1c3}
\nu\lessapprox\frac{\sigma_T}{\sigma}|\mathbb{T}|.
\end{align}

\end{itemize}
\end{itemize}

\begin{itemize}

\item[$\rhd$]
4thly, we combine all our ingredients and finish our proof of Theorem \ref{theorem 3.1}.
\begin{itemize}
	\item By \eqref{3.1c2} and the decoupling as well as H\"{o}lder's inequality, we have that if $Q_j\subset Y'$ then
\begin{align*}
\left\|e^{it\Phi}f\right\|_{L^q(Q_j)}
\lessapprox \nu^{\frac{1}{n+2}}\left(\sum_{T\in\mathbb{T}:Q_j\subset Y_T}
\left\|e^{it\Phi}f_T\right\|^q_{L^q(Q_j)}\right)^{\frac{1}{q}}.
\end{align*}

\item Via making a sum over $Q_j\subset Y'$ and using our inductive hypothesis at scale $R^{\frac{1}{2}}$, we obtain
\begin{align*}
\left\|e^{it\Phi}f\right\|^q_{L^q(Y')}
\lessapprox\nu^{\frac{2}{n}}\sum_{T\in\mathbb{T}}\left\|e^{it\Phi}f_T\right\|^q_{L^q(Y_T)}\lessapprox\nu^{\frac{2}{n}}\sum_{T\in\mathbb{T}}\left(\sigma_T^{-\frac{1}{n+2}}\|f_T\|_{L^2(\mathbb{R}^n)}\right)^q=\nu^{\frac{2}{n}}\sum_{T\in\mathbb{T}}\sigma_T^{-\frac{2}{n}}\|f_T\|^q_{L^2(\mathbb{R}^n)}.
\end{align*}

\item For each $Q_j\subset Y$,
since
$$\|e^{it\Phi}f\|_{L^q(Q_j)}~\textrm{is}~ \textrm{essentially}~\textrm{constant}~\textrm{in}~ j~\textrm{and}~|Y'|\approx|Y|,$$
we get
$$\|e^{it\Phi}f\|_{L^q(Y)}\approx\|e^{it\Phi}f\|_{L^q(Y')},
$$
thereby utilizing \eqref{3.1c3} and the fact that $\|f_T\|_{L^2(\mathbb{R}^n)}$ is essentially constant among all $T\in \mathbb{T}$ to derive
\begin{align*}
\left\|e^{it\Phi}f\right\|^q_{L^q(Y)}
&\approx\left\|e^{it\Phi}f\right\|^q_{L^q(Y')}\\
&\lessapprox\nu^{\frac{2}{n}}\sum_{T\in\mathbb{T}}\sigma_T^{-\frac{2}{n}}\left\|f_T\right\|_{L^2(\mathbb{R}^n)}^q\\
&\lessapprox\sigma^{-\frac{2}{n}}|\mathbb{T}|^{\frac{2}{n}}\sum_{T\in\mathbb{T}}\left\|f_T\right\|_{L^2(\mathbb{R}^n)}^q\\
&\thicksim\sigma^{-\frac{2}{n}}\left(\sum_{T\in\mathbb{T}}\left\|f_T\right\|^2_{L^2(\mathbb{R}^n)}\right)^{\frac{n+2}{n}}\\
&\leq\sigma^{-\frac{2}{n}}\left\|f\right\|^q_{L^2(\mathbb{R}^n)}.
\end{align*}
Taking the $q$-th root in the last estimation produces
$$\left\|e^{it\Phi}f\right\|_{L^q(Y)}
\lessapprox\sigma^{-\frac{1}{n+2}}\left\|f\right\|_{L^2(\mathbb{R}^n)}\ \  \&  \ \  Y=\cup_j Q_j.$$
\end{itemize}
\end{itemize}
\end{proof}

 Moreover, Theorem \ref{theorem 3.1} can be extended to the following form which can be verified via \cite{DGLZ} and Theorem \ref{theorem 3.1}.

\begin{theorem}\label{theorem 3.4}(Multilinear refined Strichartz estimate in dimension $n+1$).
For $2\leq k\leq n+1\ \&\ 1\le i\le k$, let $f_i: \mathbb{R}^n\rightarrow\mathbb{C}$ have frequencies $k$-transversely supported in $\mathbb B^n$ - i.e. -
$$
1\lesssim |\wedge_{i=1}^k G(\xi_i)|\ \ \&\ \
G(\xi_i)=\frac{(-2\xi_i,1)}{|(-2\xi_i,1)|}\in \mathbb{S}^n\ \ \forall\ \ \xi_i \in \text{supp}\widehat{f_i}.
$$
Suppose that $Q_1, Q_2,\cdot\cdot\cdot,Q_N$ are lattice $R^{\frac{1}{2}}$-cubes in $B^{n+1}(0,R)$ so that each
$\|e^{it(-\Delta)^\alpha}f_i \|_{L^q(Q_j)}$ is essentially dyadically constant in $j$. If
$Y=\cup_{j=1}^NQ_j$ and $\epsilon>0$, then
$$\left\|\prod_{i=1}^k \left|e^{it(-\Delta)^\alpha}f_i \right|^{\frac{1}{k}}\right\|_{L^q(Y)}
\leq C_\epsilon R^\epsilon N^{-\frac{k-1}{k(n+2)}}\prod_{i=1}^k\|f_i\|_{L^2(\mathbb{R}^n)}^{\frac{1}{k}}.$$
\end{theorem}

\begin{proof}[Proof of Theorem \ref{theorem 2.1} - the broad case]
In the broad case, there are $\geq\frac{M}{2}$ many broad $K^2$-cubes $B$. Denote the collection of $(n+1)$-tuple of transverse caps by $\Gamma$:
$$\Gamma=\big\{\tilde{\tau}=(\tau_1,\cdot\cdot\cdot,\tau_{n+1}): \tau_j\in \mathcal{S} \ \ \& \ \ \eqref{2.1b1}~\textrm{holds}~\textrm{for}~\textrm{any}~v_j \in G(\tau_j)\big\}.$$
Then for each $B\in Y_{\text{broad}}$,
$$\left\|e^{it(-\Delta)^\alpha}f \right\|^p_{L^p(B)}
\leq K^{O(1)}\prod_{j=1}^{n+1} \left(\int_B \left|e^{it(-\Delta)^\alpha}f_{\tau_j}  \right|^p \right)^{\frac{1}{n+1}}\ \ \text{for some}\ \
 \tilde{\tau}=(\tau_1,\cdot\cdot\cdot,\tau_{n+1})\in \Gamma.
 $$

 In order to exploit the transversality and make good use of  the locally constant property, we break $B$ into small balls as follows.

 \begin{itemize}
\item[$\rhd$] We cover $B=B^{n+1}(c(B), K^2)$ by cubes $B=B^{n+1}(c(B)+v, 2)$, where $v\in B^{n+1}(0,K^2)\cap \mathbb Z^{n+1}$. By the locally constant property, we can choose
$v_j\in B^{n+1}(0,K^2)\cap \mathbb{Z}^{n+1}$ such that $\|e^{it(-\Delta)^\alpha}f_{\tau_j}\|_{L^\infty(B)}$ is attained in $B^{n+1}(c(B)+v_j,2)$, and writing
$$v_j=(x_j,t_j) \ \ \& \ \ \widehat{f_{\tau_j,v_j}}(\xi)=\widehat{f_{\tau_j}}(\xi) e^{i(x_j\cdot\xi+t_j|\xi|^{2\alpha} )},
$$
we deduce that
$$e^{it(-\Delta)^\alpha}f_{\tau_j,v_j}(x)=e^{i(t+t_j)(-\Delta)^{\alpha}}f_{\tau_j}(x+x_j)
$$
and $|e^{it(-\Delta)^\alpha}f_{\tau_j,v_j}(x)|$ reaches $\|e^{it(-\Delta)^\alpha}f_{\tau_j}\|_{L^\infty(B)}$ in $B^{n+1}(c(B),2)$. Therefore
$$\int_B \left|e^{it(-\Delta)^\alpha}f_{\tau_j}  \right|^p
\leq K^{O(1)} \int_{B^{n+1}(c(B),2)} \left|e^{it(-\Delta)^\alpha}f_{\tau_j,v_j}  \right|^p.$$

\item[$\rhd$] Now for each broad $B$, we find some
$$\tilde{\tau}=(\tau_1,\cdot\cdot\cdot,\tau_{n+1})\in\Gamma \ \ \& \ \ \tilde{v}=(v_1,\cdot\cdot\cdot,v_{n+1})$$
such that
\begin{equation}\label{2.1b2}
\left\|e^{it(-\Delta)^\alpha}f \right\|^p_{L^p(B)}
\leq K^{O(1)}\prod_{j=1}^{n+1} \left(\int_{B^{n+1}(c(B),2)} \left|e^{it(-\Delta)^\alpha}f_{\tau_j,v_j}  \right|^p \right)^{\frac{1}{n+1}}\leq K^{O(1)} \int_{B^{n+1}(c(B),2)} \prod_{j=1}^{n+1} \left|e^{it(-\Delta)^\alpha}f_{\tau_j,v_j}  \right|^{\frac{p}{n+1}}.
\end{equation}

\item[$\rhd$] Since $\#\{\tilde{\tau}\}\lesssim K^{O(1)} \ \ \& \ \ \#\{\tilde{v}\}\lesssim K^{O(1)}$,
we can choose some $\tilde{\tau}$ and $\tilde{v}$ such that \eqref{2.1b2} holds for $\geq K^{-C}M$ broad balls $B$. Next we fix $\tilde{\tau}$ and $\tilde{v}$, and let $f_{\tau_j,v_j}=f_j$. After that we further sort the collection $\mathcal{B}$ of remaining broad balls as follows:
\begin{itemize}

\item For a dyadic number $A$, let $$\mathcal{B}_A=\left\{B: B\in\mathcal{B}~\textrm{and}~\textrm{for}~\textrm{each}~B~\textrm{we}~\textrm{have}
\left\|\prod_{j=1}^{n+1} \left|e^{it(-\Delta)^\alpha}f_j  \right|^{\frac{1}{n+1}}\right\|_{L^\infty(B^{n+1}(c(B),2))}\thicksim A \right\}.
$$

\item Fix $A$, for dyadic numbers $\tilde{\lambda}_{l_1,\cdot\cdot\cdot,l_{n+1}}$, let
$\mathcal{B}_{A,\tilde{\lambda}_{l_1,\cdot\cdot\cdot,l_{n+1}}}$ consist of all $ B\in\mathcal{B}_A$ for which $R^{\frac12}$-cube $Q\supset B$ contains $\thicksim \tilde{\lambda}$ cubes from $\mathcal{B}_{A}$ and obeys $\left\|e^{it(-\Delta)^\alpha}f_j \right\|_{L^q(Q)}\thicksim l_j$ for $j=1,2,\cdot\cdot\cdot,n+1$.

\end{itemize}

\item[$\rhd$] Without loss of generality, we may assume  $\|f\|_{L^2(\mathbb{R}^n)}=1$ and we can also assume all the above dyadic numbers  are between $R^{-C}$ and $R^{C}$, where $C$ is a large constant. Therefore, there exist some dyadic numbers $A,\tilde{\lambda}_{l_1,\cdot\cdot\cdot,l_{n+1}}$ such that
$\#\mathcal{B}_{A,\tilde{\lambda}_{l_1,\cdot\cdot\cdot,l_{n+1}}}\geq K^{-C}M$. Fix $A,\tilde{\lambda}_{l_1,\cdot\cdot\cdot,l_{n+1}}$ and set $\mathcal{B}_{A,\tilde{\lambda}_{l_1,\cdot\cdot\cdot,l_{n+1}}}=\mathcal{B}$. Then, by \eqref{2.1b2} and the definition of $\mathcal{B}_{A}$, we have
\begin{align}\label{2.1b3}
\left\|e^{it(-\Delta)^\alpha}f \right\|_{L^p(Y)}
&\leq K^{O(1)} \left\|\prod_{j=1}^{n+1} \left|e^{it(-\Delta)^\alpha}f_j  \right|^{\frac{1}{n+1}}\right\|_{L^p(\cup_{B\in \mathcal{B}}B^{n+1}(c(B),2))}\\
&\leq K^{O(1)} M^{\frac{1}{p}-\frac{1}{q}}\left\|\prod_{j=1}^{n+1} \left|e^{it(-\Delta)^\alpha}f_j  \right|^{\frac{1}{n+1}}\right\|
_{L^q(\cup_{B\in \mathcal{B}}B^{n+1}(c(B),2))}\nonumber\\
&\leq K^{O(1)} M^{-\frac{1}{(n+1)(n+2)}}\left\|\prod_{j=1}^{n+1} \left|e^{it(-\Delta)^\alpha}f_j  \right|^{\frac{1}{n+1}}\right\|
_{L^q(\cup_{Q\in \mathcal{Q}}Q)},\nonumber
\end{align}
where $\mathcal{Q}=\{Q: ~\text{the}~\text{relevant}~R^{\frac{1}{2}}-\text{cubes}~Q~\text{defining}~\mathcal{B}\}$.
 Note that
$$
\begin{cases} (\#\mathcal{Q})\lambda\geq(\#\mathcal{Q})\tilde{\lambda}
\thicksim\#\mathcal{B}\geq K^{-C}M;\\ \tilde{N}=\#\mathcal{Q}\geq \frac{K^{-C}M}{\lambda}.
\end{cases}
$$
So, by Theorem \ref{theorem 3.4}, we get
$$\left\|\prod_{j=1}^{n+1} \left|e^{it(-\Delta)^\alpha}f_j  \right|^{\frac{1}{n+1}}\right\|
_{L^q(\cup_{Q\in \mathcal{Q}}Q)}
\leq K^{O(1)} \left(\frac{M}{\lambda}\right)^{-\frac{n}{(n+1)(n+2)}} \|f\|_{L^2(\mathbb{R}^n)},$$
thereby getting via \eqref{2.1b3},
\begin{align*}
\left\|e^{it(-\Delta)^\alpha}f \right\|_{L^p(Y)}
\leq K^{O(1)} M^{-\frac{1}{(n+1)(n+2)}} K^{O(1)} \left(\frac{M}{\lambda}\right)^{-\frac{n}{(n+1)(n+2)}} \|f\|_{L^2(\mathbb{R}^n)}\leq K^{O(1)} M^{-\frac{1}{n+2}} \lambda^{\frac{n}{(n+1)(n+2)}}\|f\|_{L^2(\mathbb{R}^n)}.
\end{align*}

\item[$\rhd$] Our goal is to prove
$$\left\|e^{it(-\Delta)^\alpha}f\right\|_{L^p(Y)}
\leq C_\epsilon M^{-\frac{1}{n+1}}\gamma^{\frac{2}{(n+1)(n+2)}}\lambda^{\frac{n}{(n+1)(n+2)}}R^{\frac{\kappa}{(n+1)(n+2)}+\epsilon}\|f\|_{L^2(\mathbb{R}^n)}.$$
So it remains to verify
\begin{equation}\label{2.1b4}
M^{-\frac{1}{n+2}} \lambda^{\frac{n}{(n+1)(n+2)}}
\leq K^{O(1)} M^{-\frac{1}{n+1}} \gamma^{\frac{2}{(n+1)(n+2)}}\lambda^{\frac{n}{(n+1)(n+2)}}R^{\frac{\kappa}{(n+1)(n+2)}+\epsilon} - \text{i.e.} -
M\leq K^{O(1)} \gamma^2 R^\kappa.
\end{equation}
However, the second equivalent inequality of \eqref{2.1b4} follows from the definition  \eqref{2.1a1} of $\gamma$ which ensures
$
M\leq \gamma R^\kappa \ \ \& \ \ \gamma\geq K^{-2\kappa}.
$
\end{itemize}
\end{proof}

\subsubsection{The narrow case.}
In order to prove the narrow case of Theorem \ref{theorem 2.1}, we have the following lemma which is essentially contained in Bourgain-Demeter \cite{BD}.

\begin{lemma}\label{lemma 3.6}
Suppose that:
\begin{itemize}
\item[\rm (i)] $B$ is a narrow $K^2$-cube in $\mathbb{R}^{n+1}$ and takes $c(B)$ as its center;

\item[\rm (ii)] $\mathcal{S}$ denotes the set of $K^{-1}$-cubes which tile $\mathbb B^n$;

\item[\rm (iii)] $\omega_B$ is a weight function which is essentially a characteristic function on $B$ - more precisely -
$$\text{supp}\widehat{\omega_B}\subset B(0,K^{-2}) \ \ \& \ \ \chi_B(\tilde{x})\lesssim \omega_B(\tilde{x})\leq \left( 1+\frac{|\tilde{x}-c(B)|}{K^2}\right)^{-1000n}.$$
\end{itemize}
Then
$$\left\|e^{it(-\Delta)^\alpha}f \right\|_{L^p(B)}
\leq C_\epsilon K^\epsilon \left( \sum_{\tau\in \mathcal{S}} \left\|e^{it(-\Delta)^\alpha}f_\tau \right\|^2_{L^p(\omega_B)} \right)^{\frac{1}{2}}\ \ \forall\ \ \epsilon>0.
$$
\end{lemma}

\begin{proof}[Proof of Theorem \ref{theorem 2.1} - the narrow case]
The main method we used is the parabolic rescaling and induction on radius. Next we prove the narrow case step by step.
	\begin{itemize}
\item [$\rhd$] 1stly, we consider the wave packet decomposition which is similar to Theorem \ref{theorem 3.1} but with different scale.
We break the physical ball $B^n(0,R)$ into $\frac{R}{K}$-cubes $D$. From \cite{T}, we have
$$f=\sum_{\tau, D} f_{T_{\tau,D}}\ \  \& \ \ f_{T_{\tau,D}}=(\hat f \chi_{\tau})^{\vee}\chi_D.$$
By computation, we have
$e^{it(-\Delta)^\alpha}f_{T_{\tau,D}}$ (whenever restricted to $B^{n+1}(0,R)$) is essentially supported on an $\frac{R}{K}\times\cdot\cdot\cdot\times\frac{R}{K}\times R$-box, denoted by
$$T_{\tau,D}=\left\{(x,t):(x,t)\in B^{n+1}(0,R)\ \ \& \ \  \left|x-c(D)-2t\alpha|c(\tau)|^{2\alpha-2}c(\tau) \right|\leq \frac{R}{K}\ \ \& \ \  0<t<R\right\}.$$
Here $c(\tau) \& c(D)$ denote the centers of $\tau \& D$ respectively.
For a fixed $\tau$, the different tubes $T_{\tau,D}$ tile $B^{n+1}(0,R)$.
Next we write
$f=\sum_T f_T$
for convenience.
Therefore, by decoupling theorem, for each narrow $K^2$-cube $B$, we have
\begin{align}\label{2.1n1}
\left\|e^{it(-\Delta)^\alpha}f\right\|_{L^p(B)}
\lesssim K^{\epsilon^4}\left(\sum_T \left\|e^{it(-\Delta)^\alpha}f_T \right\|^2_{L^p(\omega_B)}\right)^{\frac{1}{2}}.
\end{align}
The reason to take $K^{\epsilon^4}$ in \eqref{2.1n1} is that there is a $\frac{1}{K^{2\epsilon}}$ satisfying $\frac{ K^{3\epsilon^4}}{ K^{2\epsilon}}\ll 1$ at the end of the proof.

\item[$\rhd$] 2ndly, we perform a dyadic pigeonholing to get our inductive hypothesis for each $f_T$. Note that
$$
\begin{cases}
K=R^\delta=R^{\epsilon^{100}};\\
R_1=\frac{R}{K^2}=R^{1-2\delta};\\
K_1=R_1^\delta=R^{\delta-2\delta^2}.\\
\end{cases}
$$
So, not only tiling the box $T$ by $KK_1^2\times\cdot\cdot\cdot\times KK_1^2\times K^2K_1^2$-tubes $S$, but also tiling the box $T$ by
$R^{\frac{1}{2}}\times\cdot\cdot\cdot\times R^{\frac{1}{2}}\times KR^{\frac{1}{2}}$-tubes $S'$ which are running parallel to the long axis of box $T$, we utilize the parabolic rescaling to reveal that the box $T$ becomes an $R_1$-cube as well as the tubes $S'$ and $S$ become lattice $R_1^{\frac{1}{2}}$-cubes and $K_1^2$-cubes respectively. See 7thly for more details.

\item[$\rhd$] 3rdly, we classify the tubes $S$ and $S'$ inside each $T$ as follows.

\begin{itemize}

\item For dyadic numbers $\eta, \beta_1$, let
$\mathbb{S}_{T,\eta,\beta_1}=\big\{S: S\subset T ~\text{each} ~\text{of}~\text{which}~\text{contains}~\thicksim \eta~\text{narrow}~K^2-\text{cubes}~\text{in}~Y_{\text{narrow}}~\text{and}~\|e^{it(-\Delta)^\alpha}f_T\|_{L^p(S)}\thicksim\beta_1\big\}$.

\item Fix $\eta, \beta_1$, and
for dyadic number $\lambda_1$, let
$\mathbb{S}_{T,\eta,\beta_1,\lambda_1}=\big\{S: S\in \mathbb{S}_{T,\eta,\beta_1} ~\text{and}~\text{the}~\text{tube}~S'\supset S~\text{contains}~\thicksim \lambda_1~\text{tubes}~\text{from}~\mathbb{S}_{T,\eta,\beta_1}\big\}$.

\item For the fixed $\eta, \beta_1, \lambda_1$, we sort the boxes $T$. For dyadic numbers $\beta_2,M_1,\gamma_1$, let
$\mathbb{B}_{\eta,\beta_1,\lambda_1,\beta_2,M_1,\gamma_1}$ denote the collection of boxes $T$ each of which satisfyies
$$\|f_T\|_{L^2(\mathbb{R}^n)}\thicksim\beta_2 \ \ \&  \ \ \#\mathbb{S}_{T,\eta,\beta_1,\lambda_1}\thicksim M_1$$
and
\begin{align}\label{2.1n1r}
\max_{T_r\subset T:r\geq  K_1^2}\frac{\#\{S: S\in\mathbb{S}_{T,\eta,\beta_1,\lambda_1}\ \ \&  \ \ S\subset T_r\}}{r^\kappa}\thicksim\gamma_1,
\end{align}
where $T_r$ are $Kr\times\cdot\cdot\cdot\times Kr\times K^2r$-tubes in $T$ which are  parallel to the long axis of $T$.
\end{itemize}

\item[$\rhd$] 4thly, let
$$Y_{T,\eta,\beta_1,\lambda_1}=\cup_{S\in \mathbb{S}_{T,\eta,\beta_1,\lambda_1}} S.$$
Then, for $Y_{\textrm{narrow}}$ we can write
\begin{align*}
e^{it(-\Delta)^\alpha}f
=\Sigma_{\eta,\beta_1,\lambda_1,\beta_2,M_1,\gamma_1}
\left(\sum_{T\in\mathbb{B}_{\eta,\beta_1,\lambda_1,\beta_2,M_1,\gamma_1} }  e^{it(-\Delta)^\alpha}f_T \cdot\chi_{Y_{T,\eta,\beta_1,\lambda_1}} \right)
+O(R^{-1000n})\|f\|_{L^2(\mathbb{R}^n)}.
\end{align*}
The error term $O(R^{-1000n})\|f\|_{L^2(\mathbb{R}^n)}$ can be neglected.
\begin{itemize}
\item In particular, on each narrow $B$ we have
\begin{align}\label{2.1n2}
e^{it(-\Delta)^\alpha}f
=\Sigma_{\eta,\beta_1,\lambda_1,\beta_2,M_1,\gamma_1}
\left(\sum_{\substack{T\in\mathbb{B}_{\eta,\beta_1,\lambda_1,\beta_2,M_1,\gamma_1}\\B\subset Y_{T,\eta,\beta_1,\lambda_1 } }} e^{it(-\Delta)^\alpha}f_T \right)
.
\end{align}

\item Without loss of generality, we assume $$
\begin{cases}
\|f\|_{L^2(\mathbb{R}^n)}=1;\\
1\leq\eta\leq K^{O(1)},R^{-10n}\leq\beta_1\leq K^{O(1)},1\leq\lambda_1\leq R^{O(1)};\\
R^{-10n}\leq\beta_2\leq 1,1\leq M_1\leq R^{O(1)},K^{-2n}\leq\gamma_1\leq R^{O(1)}.
\end{cases}
$$
Therefore, there are only $O(\log R)$  significant choices for each dyadic number.

\item By \eqref{2.1n2}, the pigeonholing and \eqref{2.1n1}, we can choose
$\eta,\beta_1,\lambda_1,\beta_2,M_1,\gamma_1 $ such that

\begin{align}\label{2.1n4}
\left\|e^{it(-\Delta)^\alpha}f\right\|_{L^p(B)}
\lesssim(\log R)^6 K^{\epsilon^4}
\left(\sum_{\substack{T\in\mathbb{B}_{\eta,\beta_1,\lambda_1,\beta_2,M_1,\gamma_1}\\B\subset Y_{T,\eta,\beta_1,\lambda_1 } }}
\left\|e^{it(-\Delta)^\alpha}f_T\right\|^2_{L^p(\omega_B)} \right)^{\frac{1}{2}}
\end{align}
holds for $\gtrsim(\log R)^{-6}$ narrow $K^2$-cubes $B$.

\end{itemize}
\item[$\rhd$] 5thly, we fix $\eta,\beta_1,\lambda_1,\beta_2,M_1,\gamma_1$ for the rest of the proof. Let
$$Y_{T,\eta,\beta_1,\lambda_1}=Y_T~\&~\mathbb{B}_{\eta,\beta_1,\lambda_1,\beta_2,M_1,\gamma_1}=\mathbb{B}.
$$
Let $Y'\subset Y_{\textrm{narrow}}$ be a union of narrow $K^2$-cubes $B$ each of which obeys \eqref{2.1n4}

and
\begin{align}\label{2.1n5}
\begin{cases}\#\{T: T\in \mathbb{B}\ \ \&  \ \ B\subset Y_T\}\thicksim\nu
\ \ \text{
for some dyadic number}\ \ 1\leq\nu\leq K^{O(1)};\\
\#\{B: B\subset Y' \ \ \& \ \ B~\textrm{are}~K^2-\textrm{cubes}\}\gtrsim(\log R)^{-7}M.
\end{cases}
\end{align}
By our assumption that
$\|e^{it(-\Delta)^\alpha}f\|_{L^p(B_k)}$ is essentially constant in $k=1,2,\cdot\cdot\cdot,M$, in the narrow case we have

\begin{align}\label{2.1n6}
\left\|e^{it(-\Delta)^\alpha}f\right\|^p_{L^p(Y)}
&\lesssim(\log R)^7
\sum_{B\subset Y'}
\left\|e^{it(-\Delta)^\alpha}f\right\|^p_{L^p(B)}.
\end{align}
For each $B\subset Y'$, it follows from \eqref{2.1n4}, H\"{o}lder's inequality and \eqref{2.1n5} that

\begin{align}\label{2.1n7}
\left\|e^{it(-\Delta)^\alpha}f\right\|^p_{L^p(B)}
&\lesssim(\log R)^{6p} K^{\epsilon^4p}
\left(   \sum_{T\in \mathbb{B}: B\subset Y_T}
\left\|e^{it(-\Delta)^\alpha}f_T\right\|^2_{L^p(\omega_B)} \right)^{\frac{p}{2}}\\
&\lesssim(\log R)^{6p} K^{\epsilon^4p} \nu^{\frac{p}{2}-1}
\sum_{T\in \mathbb{B}: B\subset Y_T}
\left\|e^{it(-\Delta)^\alpha}f_T\right\|^p_{L^p(\omega_B)}.\nonumber
\end{align}
Via \eqref{2.1n6} and \eqref{2.1n7},
we have
\begin{align}\label{2.1n8}
\left\|e^{it(-\Delta)^\alpha}f\right\|_{L^p(Y)}
&\lesssim(\log R)^{\frac{7}{p}}
\left(\sum_{B\subset Y'}
\left\|e^{it(-\Delta)^\alpha}f\right\|^p_{L^p(B)}\right)^{\frac{1}{p}}\\\nonumber
&\lesssim(\log R)^{\frac{7}{p}}
\left(\sum_{B\subset Y'}(\log R)^{6p} K^{\epsilon^4p} \nu^{\frac{p}{2}-1}\sum_{T\in \mathbb{B}: B\subset Y_T}
\left\|e^{it(-\Delta)^\alpha}f_T\right\|^p_{L^p(\omega_B)}
\right)^{\frac{1}{p}}\\\nonumber
&\lesssim(\log R)^{13} K^{\epsilon^4} \nu^{\frac{1}{n+1}}
\left(\sum_{B\subset Y'} \sum_{T\in \mathbb{B}: B\subset Y_T}
\left\|e^{it(-\Delta)^\alpha}f_T\right\|^p_{L^p(\omega_B)}\right)^{\frac{1}{p}}\\\nonumber
&\lesssim(\log R)^{13} K^{\epsilon^4} \nu^{\frac{1}{n+1}}
\left(\sum_{T\in \mathbb{B}}
\left\|e^{it(-\Delta)^\alpha}f_T\right\|^p_{L^p(Y_T)}\right)^{\frac{1}{p}}.\nonumber
\end{align}

\item[$\rhd$] 6thly, regarding each $\|e^{it(-\Delta)^\alpha}f_T\|_{L^p(Y_T)}$, we apply the parabolic rescaling and induction on radius. For each $K^{-1}$-cube $\tau=\tau_T$ in $\mathbb B^n$, we write $\xi=\xi_0+K^{-1}\eta\in\tau$, where $\xi_0=c(\tau)$.
Similarly to the argument of \eqref{3.1b1}, we also consider a collection of the normalized phase functions
$$\mathcal{NPF}(L,c_0)=\left\{\Phi\in C_0^\infty(B^n(0,2)):\left \|\Phi(\xi)-\frac{|\xi|^2}{2}\right\|_{C^L(\mathbb B^n)}\leq  c_0\right\}.$$
Via the similar parabolic rescaling,
\begin{equation*}
\left\{
\begin{aligned}
 &\tilde{x}=K^{-1} H^{-t}( x+t\nabla\Phi(\xi_0));\\
&\tilde{t}=K^{-2}t,\\
\end{aligned}
\right.
\end{equation*}
we reach
\begin{align}\label{2.1n9}
\|e^{it\Phi}f_T(x)\|_{L^p(Y_T)}
=K^{-\frac{1}{n+1}}|H|^{-\frac{1}{n+1}}\|e^{i\tilde{t}\Phi_{K^{-1},\xi_0}}g(\tilde{x})\|_{L^p(\tilde{Y})}
\thicksim K^{-\frac{1}{n+1}} \|e^{i\tilde{t}\Phi_{K^{-1},\xi_0}}g(\tilde{x})\|_{L^p(\tilde{Y})},
\end{align}
where
\begin{equation*}
\left\{
\begin{aligned}
&|H|\thicksim 1~(\text{since}~ |\xi|\thicksim 1);\\
&\text{supp} \hat{g}\subset \mathbb B^n;\\
&\|g\|_{L^2(\mathbb{R}^n)}=\|f_T\|_{L^2(\mathbb{R}^n)},\\
\end{aligned}
\right.
\end{equation*}
as well as $\tilde{Y}$ is the image of $Y_T$ under the new coordinates and $\Phi_{K^{-1},\xi_0}$ is similar to \eqref{NPF}.

\item[$\rhd$] 7thly, we apply inductive hypothesis \eqref{2.1a2} (replacing $(-\Delta)^{\alpha}$ with $\Phi$ ) at scale $R_1=\frac{R}{K^2}$ to $\|e^{i\tilde{t}(-\Delta)^{\alpha}}g(\tilde{x})\|_{L^p(\tilde{Y})}$ with $M_1,\gamma_1,\lambda_1,R_1$. Under parabolic rescaling, the relation between preimage and image is as follows:
\begin{equation*}
\left\{
\begin{aligned}
&T~\left(\frac{R}{K}\times\cdot\cdot\cdot\times\frac{R}{K}\times R -\text{tube}\right) \longrightarrow \tilde{T}~\left(R_1-\text{cube}\right);\\
&S'~\left(R^{\frac{1}{2}}\times\cdot\cdot\cdot\times R^{\frac{1}{2}}\times KR^{\frac{1}{2}}-\text{tube}\right)\longrightarrow \tilde{S'}~\left(R_1^{\frac12}-\text{cube}\right);\\
&S~\left(KK_1^2\times\cdot\cdot\cdot\times KK_1^2\times K^2K_1^2-\text{tube}\right)\longrightarrow \tilde{S}~\left(K_1^2-\text{cube}\right).\\
\end{aligned}
\right.
\end{equation*}
More precisely, we have
$$\#\{\tilde{S}: \tilde{S}\subset \tilde{T}\ \ \& \ \ \tilde{S}\subset \tilde{Y} \}\thicksim M_1$$
and the $K_1^2$-cubes $\tilde{S}$ are organized into $R_1^{\frac{1}{2}}$-cubes $\tilde{S'}$ such that
$$\#\{\tilde{S}: \tilde{S} \subset \tilde{S'}\}\thicksim \lambda_1.$$
Moreover, $\|e^{i\tilde{t}(-\Delta)^{\alpha}}g(\tilde{x})\|_{L^p(\tilde{S})}$ is dyadically a constant in $S\subset Y_T$. By our choice of $\gamma_1$, we have
$$\max_{\substack{B^{n+1}(x',r)\subset \tilde{T}\\x'\in \mathbb{R}^{n+1},r\geq  K_1^2}}
\frac{\#\{\tilde{S}: \tilde{S}\subset B^{n+1}(x',r)\}}{r^\kappa}\thicksim\gamma_1.$$
Hence, by the
inductive hypothesis \eqref{2.1a2} (replacing $(-\Delta)^{\alpha}$ with $\Phi$) at scale $R_1$, we have
\begin{equation*}\label{2.1n10.}
\|e^{i\tilde{t}\Phi_{K^{-1},\xi_0}}g(\tilde{x})\|_{L^p(\tilde{Y})}
\lesssim
M_1^{-\frac{1}{n+1}}\gamma_1^{\frac{2}{(n+1)(n+2)}}\lambda_1^{\frac{n}{(n+1)(n+2)}}\left(\frac{R}{K^2}\right)^{\frac{\kappa}{(n+1)(n+2)}+\epsilon}
\|g\|_{L^2(\mathbb{R}^n)}.
\end{equation*}
 By \eqref{2.1n9} and $\|g\|_{L^2(\mathbb{R}^n)}=\|f_T\|_{L^2(\mathbb{R}^n)}$, we get
\begin{align}\label{2.1n10..}
\|e^{it\Phi}f_T(x)\|_{L^p(Y_T)}
&\lesssim K^{-\frac{1}{n+1}}
M_1^{-\frac{1}{n+1}}\gamma_1^{\frac{2}{(n+1)(n+2)}}\lambda_1^{\frac{n}{(n+1)(n+2)}}\left(\frac{R}{K^2}\right)^{\frac{\kappa}{(n+1)(n+2)}+\epsilon}
\|f_T\|_{L^2(\mathbb{R}^n)}.
\end{align}
Since \eqref{2.1n10..} also holds whenever replacing $\Phi$ with $(-\Delta)^{\alpha}$, we get
\begin{align}\label{2.1n10}
\|e^{it(-\Delta)^\alpha}f_T(x)\|_{L^p(Y_T)}
&\lesssim K^{-\frac{1}{n+1}}
M_1^{-\frac{1}{n+1}}\gamma_1^{\frac{2}{(n+1)(n+2)}}\lambda_1^{\frac{n}{(n+1)(n+2)}}\left(\frac{R}{K^2}\right)^{\frac{\kappa}{(n+1)(n+2)}+\epsilon}
\|f_T\|_{L^2(\mathbb{R}^n)}.
\end{align}
By \eqref{2.1n8} and \eqref{2.1n10}, we obtain
\begin{align}\label{2.1n11}
\|e^{it(-\Delta)^\alpha}f\|_{L^p(Y)}
&\lesssim(\log R)^{13} K^{\epsilon^4} \nu^{\frac{1}{n+1}}
\left(\sum_{T\in \mathbb{B}}
  \left(K^{-\frac{1}{n+1}}
M_1^{-\frac{1}{n+1}}\gamma_1^{\frac{2}{(n+1)(n+2)}}\lambda_1^{\frac{n}{(n+1)(n+2)}}\left(\frac{R}{K^2}\right)^{\frac{\kappa}{(n+1)(n+2)}+\epsilon}
\|f_T\|_{L^2(\mathbb{R}^n)}\right)^p\right)^{\frac{1}{p}}\\\nonumber
&\lesssim K^{2\epsilon^4} \nu^{\frac{1}{n+1}} K^{-\frac{1}{n+1}}
M_1^{-\frac{1}{n+1}}\gamma_1^{\frac{2}{(n+1)(n+2)}}\lambda_1^{\frac{n}{(n+1)(n+2)}}\left(\frac{R}{K^2}\right)^{\frac{\kappa}{(n+1)(n+2)}+\epsilon}
\left(\sum_{T\in \mathbb{B}}\|f_T\|^p_{L^2(\mathbb{R}^n)}\right)^{\frac{1}{p}}\\
&\lesssim K^{2\epsilon^4} \left(\frac{\nu}{\# \mathbb{B}}\right)^{\frac{1}{n+1}} K^{-\frac{1}{n+1}}
M_1^{-\frac{1}{n+1}}\gamma_1^{\frac{2}{(n+1)(n+2)}}\lambda_1^{\frac{n}{(n+1)(n+2)}}\left(\frac{R}{K^2}\right)^{\frac{\kappa}{(n+1)(n+2)}+\epsilon}
\|f\|_{L^2(\mathbb{R}^n)},\nonumber
\end{align}
where the third inequality follows from the assumption that
$\|f_T\|_{L^2(\mathbb{R}^n)}$ is essentially constant in $T\in \mathbb{B}$ and
then implies
$$\left(\sum_{T\in \mathbb{B}}\|f_T\|^p_{L^2(\mathbb{R}^n)}\right)^{\frac{1}{p}}
\leq \left(\frac{1}{\# \mathbb{B}}\right)^{\frac{1}{n+1}} \left(\sum_{T}\|f_T\|^2_{L^2(\mathbb{R}^n)}\right)^{\frac12}
\lesssim\left(\frac{1}{\# \mathbb{B}}\right)^{\frac{1}{n+1}} \|f\|_{L^2(\mathbb{R}^n)}.$$

\item[$\rhd$] 8thly, we consider the lower bound and the upper bound of
$$\#\{(T,B):T\in \mathbb{B} \ \ \& \ \ B\subset Y_T\cap Y'\}.$$
\begin{itemize}

\item On the one hand, by the definition of $\nu$ as in \eqref{2.1n5}, there is a lower bound
$$\#\{(T,B):T\in \mathbb{B}\ \ \& \ \ B\subset Y_T\cap Y'\}\gtrsim (\log R)^{-7} M\nu.$$

\item On the other hand, by our choices of $M_1$ and $\eta$, for each $T\in \mathbb{B}$,
\begin{equation*}
\left\{
\begin{aligned}
&\#\{S: S\subset Y_T \}\thicksim M_1;\\
&\#\{B: B\subset S  \ \ \& \ \ B\subset  Y_{\text{narrow}}\}\thicksim \eta.\\
\end{aligned}
\right.
\end{equation*}
so
$$\#\{(T,B):T\in \mathbb{B}\ \ \& \ \ B\subset Y_T\cap Y'\}\lesssim(\# \mathbb{B})M_1 \eta.$$
\end{itemize}
Therefore, we get
\begin{align}\label{2.1n11v}
\frac{\nu}{\# \mathbb{B}}\lesssim \frac{(\log R)^7 M_1 \eta}{M}.
\end{align}

\item[$\rhd$] 9thly, we want to obtain the relation between $\gamma$ and $\gamma_1$. By our choices of $\gamma_1$ as in \eqref{2.1n1r} and $\eta$,
\begin{align*}\label{2.1n12}
\gamma_1 \cdot\eta
&\thicksim \max_{T_r\subset T:r\geq  K_1^2}
\frac{\#\{S: S\subset Y_T\cap T_r\}}{r^\kappa}\cdot\#\{B: B\subset S\cap Y_{\textrm{narrow}}~\textrm{for}~\textrm{any}~\textrm{fixed}~S\subset Y_T\}\\
&\lesssim \max_{T_r\subset T:r\geq  K_1^2}
\frac{\#\{B: B\subset Y\ \ \& \ \  B\subset T_r\}}{r^\kappa}\\
&\leq\frac{K\gamma(Kr)^\kappa}{r^\kappa}\nonumber\\
&=\gamma K^{\kappa+1}.
\end{align*}
Hence,
\begin{equation}\label{2.1n13}
\eta\lesssim\frac{\gamma K^{\kappa+1}}{\gamma_1}.
\end{equation}

\item[$\rhd$] 10thly, we complete the proof of Theorem \ref{theorem 2.1}.
\begin{itemize}
	
\item On the one hand,
\begin{equation*}
\left\{
\begin{aligned}
&\#\{S: S\subset S'  \ \ \& \ \ S\subset Y_T \}\thicksim \lambda_1;\\
&\#\{B: B\subset S  \ \ \& \ \ B\subset  Y_{\text{narrow}}\}\thicksim \eta.\\
\end{aligned}
\right.
\end{equation*}

\item On the other hand,  we can cover $S'$ by $\thicksim K$ finitely overlapping $R^{\frac{1}{2}}$-balls and each $R^{\frac{1}{2}}$-ball contains $\lesssim\lambda$ many $K^2$-cubes in $Y$.
\end{itemize}

Thus it follows that
\begin{align}\label{2.1n14}
\lambda_1
\lesssim \frac{K\lambda}{\eta}.
\end{align}
Inserting \eqref{2.1n11v}, \eqref{2.1n14}  and \eqref{2.1n13} into \eqref{2.1n11} gives
\begin{align*}\label{2.1n15}
\|e^{it(-\Delta)^\alpha}f\|_{L^p(Y)}
&\lesssim K^{2\epsilon^4} \left(\frac{(\log R)^7 M_1 \eta}{M}\right)^{\frac{1}{n+1}} K^{-\frac{1}{n+1}}
M_1^{-\frac{1}{n+1}}\gamma_1^{\frac{2}{(n+1)(n+2)}}\left( \frac{K\lambda}{\eta}\right)^{\frac{n}{(n+1)(n+2)}}\left(\frac{R}{K^2}\right)^{\frac{\kappa}{(n+1)(n+2)}+\epsilon}
\|f\|_{L^2(\mathbb{R}^n)}\\\nonumber
&\lesssim \frac{ K^{3\epsilon^4}}{ K^{2\epsilon}}  \left(\frac{\eta\gamma_1}{K^{\kappa+1}}\right)^{\frac{2}{(n+1)(n+2)}}
 M^{-\frac{1}{n+1}}  \lambda^{\frac{n}{(n+1)(n+2)}}  {R}^{\frac{\kappa}{(n+1)(n+2)}+\epsilon} \|f\|_{L^2(\mathbb{R}^n)}\\
&\lesssim  \frac{ K^{3\epsilon^4}}{ K^{2\epsilon}}  M^{-\frac{1}{n+1}} \gamma^{\frac{2}{(n+1)(n+2)}}
  \lambda^{\frac{n}{(n+1)(n+2)}}  {R}^{\frac{\kappa}{(n+1)(n+2)}+\epsilon} \|f\|_{L^2(\mathbb{R}^n)}.\nonumber
\end{align*}
where the last inequality follows from \eqref{2.1n13}. It is not hard to see that
$\frac{ K^{3\epsilon^4}}{ K^{2\epsilon}}\ll 1$ and the induction concludes the argument for the narrow case.
\end{itemize}
\end{proof}



\bigskip

\noindent  Dan Li

\smallskip

\noindent  Laboratory of Mathematics and Complex Systems
(Ministry of Education of China),
School of Mathematical Sciences, Beijing Normal University,
Beijing 100875, People's Republic of China

\smallskip

\noindent {\it E-mails}: \texttt{danli@mail.bnu.edu.cn}

\bigskip

\noindent Junfeng Li(Corresponding author)

\smallskip

\noindent School of Mathematical Sciences, Dalian University of Technology, Dalian, LN, 116024, China

\smallskip

\noindent{\it E-mail}: \texttt{junfengli@dlut.edu.cn}

\bigskip

\noindent Jie Xiao

\smallskip

\noindent Department of Mathematics and Statistics,
		Memorial University, St. John's, NL A1C 5S7, Canada

\smallskip

\noindent{\it E-mail}: \texttt{jxiao@math.mun.ca}


\begin{thebibliography}{10}

	
\bibitem{BBCR} J. A. Barcelo, J. Bennett, A. Carbery and K. M. Rogers, On the dimension of divergence sets of dispersive equations. Math. Ann. 349 (2011), no. 3, 599-622.
\vspace{-.3cm}

\bibitem{B3} J. Bourgain, Some new estimates on oscillatory integrals. Essays on Fourier analysis in Honor of Elias M. Stein (Princeton, NJ, 1991), Princeton Math. Ser., vol. 42, Princeton University Press, New Jersey, 1995, pp. 83-112.
\vspace{-.3cm}
\bibitem{B1} J. Bourgain, On the Schr\"{o}dinger maximal function in higher dimension. Proc. Steklov Inst. Math. 280 (2013), no. 1, 46-60.
\vspace{-.3cm}
\bibitem{B2} J. Bourgain, A note on the Schr\"{o}dinger maximal function. J. Anal. Math. 130 (2016), 393-396.
\vspace{-.3cm}
\bibitem{BD} J. Bourgain and C. Demeter, The proof of the $l^2$ decoupling conjecture. Ann. of Math. (2) 182
(2015), no. 1, 351-389.
\vspace{-.3cm}

\bibitem{C} L. Carleson, Some analytic problems related to statistical mechanics, in: Euclidean harmonic analysis (Proc. Sem., Univ. Maryland., College Park, Md., 1979). pages 5-45, Lecture Notes in Math., 779, Springer, Berlin, 1980.
\vspace{-.3cm}
\bibitem{CK} C. H. Cho and H. Ko, A note on maximal estimates of generalized  Schr\"{o}dinger equation. arXiv: 1809. 03246v1.
\vspace{-.3cm}
\bibitem{CLV} C. H. Cho, S. Lee and A. Vargas, Problems on pointwise convergence of solutions to the Schr\"{o}dinger equation. J. Fourier Anal. Appl. 18 (2012), no. 5, 972-994.
\vspace{-.3cm}
\bibitem{COX} Y. Cho, T. Ozawa and S. Xia, Remarks on some dispersive estimates.
Commun. Pure Appl. Anal. 10 (2011), no. 4, 1121-1128.
\vspace{-.3cm}
\bibitem{DK} B. E. J. Dahlberg and C. E. Kenig, A note on the almost everywhere behavior of solutions to the Schr\"{o}dinger equation, in: Harmonic analysis (Minneapolis, Minn., 1981). pages 205-209, Lecture Notes in Math., 908, Springer, Berlin-New York, 1982.
\vspace{-.3cm}
\bibitem{Dinh} V. D. Dinh, Strichartz estimates for the fractional Schr\"{o}dinger and wave equations on compact manifolds without boundary. J. Differential Equations 263 (2017), no. 12, 8804-8837.

\vspace{-.3cm}
\bibitem{DGL} X. Du, L. Guth and X. Li, A sharp Schr\"{o}dinger maximal estimate in $\mathbb{R}^2$. Ann. of Math. (2) 186 (2017), no. 2, 607-640.
\vspace{-.3cm}
\bibitem{DGLZ} X. Du, L. Guth, X. Li and R. Zhang, Pointwise convergence of Schr\"{o}dinger solutions and multilinear refined Strichartz estimates.
Forum Math. Sigma. 6 (2018), e14, 18 pp.
\vspace{-.3cm}

\bibitem{DZ} X. Du and R. Zhang, Sharp $L^2$ estimate  of Schr\"{o}dinger maximal function in higher dimensions. arXiv:1805. 02775v1.
\vspace{-.3cm}



\bibitem{L} S. Lee, On pointwise convergence of the solutions to Schr\"{o}dinger equations in $\mathbb{R}^2$. Int. Math. Res. Not. 2006, Art. ID 32597, 21 pp.
\vspace{-.3cm}
\bibitem{LR} S. Lee and K. Rogers, The Schr\"{o}dinger equation along curves and the quantum harmonic oscillator. Adv. Math. 229 (2012), no. 3, 1359-1379.
\vspace{-.3cm}
\bibitem{LR3} R. Luc\`{a} and K. Rogers, Average decay for the Fourier transform of measures with applications. J. Eur. Math. Soc. 21 (2019), no. 2, 465-506.
\vspace{-.3cm}
\bibitem{LR1} R. Luc\`{a} and K. Rogers, Coherence on fractals versus pointwise convergence for the Schr\"{o}dinger equation. Comm. Math. Phys. 351 (2017), no.1, 341-359.
\vspace{-.3cm}
\bibitem{LR2} R. Luc\`{a} and K. Rogers, A note on pointwise convergence for the Schr\"{o}dinger equation. Math. Proc. Cambridge Philos. Soc. 166 (2019), no. 2, 209-218.
\vspace{-.3cm}
\bibitem{MYZ} C. Miao, J. Yang and J. Zheng, An improved maximal inequality for 2D fractional order Schr\"{o}dinger operators. Studia Math. 230 (2015), no. 2, 121-165.
\vspace{-.3cm}

\bibitem{MVV} A. Moyua, A. Vargas and L. Vega, Schr\"{o}dinger maximal function and restriction properties of the Fourier transform. Internat. Math. Res. Notices (1996), no. 16, 793-815.
\vspace{-.3cm}

\bibitem{SS} P. Sj\"{o}gren and P. Sj\"{o}lin, Convergence properties for the time-dependent Schr\"{o}dinger equation. Ann. Acad. Sci. Fenn. Ser. A I Math. 14 (1989), no. 1, 13-25.
\vspace{-.3cm}
\bibitem{S1} P. Sj\"{o}lin, Regularity of solutions to the Schr\"{o}dinger equation. Duke Math. J. 55 (1987), no. 3, 699-715.
\vspace{-.3cm}
\bibitem{S3} P. Sj\"{o}lin, Nonlocalization of operators of Schr\"{o}dinger type. Ann. Acad. Sci. Fenn. Math. 38 (2013), no. 1, 141-147.
\vspace{-.3cm}




\bibitem{T} T, Tao, A sharp bilinear restrictions estimate for paraboloids. Geom. Funct. Anal. 13 (2003), no. 6, 1359-1384.
\vspace{-.3cm}

\bibitem{TV} T. Tao and A. Vargas, A bilinear approach to cone multipliers. II. Applications. Geom. Funct. Anal. 10 (2000), no. 1,  216-258.

\vspace{-.3cm}

\bibitem{V2} L. Vega, E1 Multiplicador de Schr\"{o}dinger, la Function Maximal y los Operadores de Restriccion(thesis). Departamento de Matematicas. Univ. Aut\'{o}noma de Madrid, Madrid (1988).
\vspace{-.3cm}
\bibitem{V1} L. Vega, Schr\"{o}dinger equations: pointwise convergence to the initial data. Proc. Amer. Math. Soc. 102 (1988), no. 4, 874-878.
\vspace{-.3cm}
\bibitem{Z} D. \v{Z}ubrini\'{c}, Singular sets of Sobolev functions. C. R. Math. Acad. Sci. Paris. 334 (2002), no. 7, 539-544.
\vspace{-.3cm}
\end{thebibliography}
\end{document}